\documentclass[10pt,a4paper]{article}

\usepackage{amsmath, amsfonts, amsthm, amssymb, amscd,enumerate}
\usepackage{graphicx, color}
\usepackage[mathcal]{euscript}
\usepackage{epstopdf}
\newtheorem{theorem}{Theorem}[section]
\newtheorem{proposition}{Proposition}[section]
\newtheorem{lemma}{Lemma}[section]

\newtheorem{nb}{Remark}[section]
\newtheorem{example}{Example}[section]

\numberwithin{equation}{section}

\newcommand{\ds}{\displaystyle}
\renewcommand{\d}{\,{\rm d}}
\newcommand{\fr}[3]{\left(\frac{\partial #1}{\partial #2}\right)_{#3}}

\newcommand{\Ti}{T_{\rm i}}

\newcommand{\qq}{\left(\frac52+\frac{\Ti}{T}\right)}
\makeatletter

\newlength{\captionwidth}
\long\def\@makecaption#1#2{%
   \vskip 10\p@
   \setbox\@tempboxa\hbox{\small #1: #2}%
   \ifdim \wd\@tempboxa > \captionwidth 
       \hbox to\hsize{\hfil
       \parbox[t]{\captionwidth}{
        \small #1: #2\par}
       \hfil}
     \else
       \hbox to\hsize{\hfil\box\@tempboxa\hfil}%
   \fi}

\makeatother

\begin{document}
\title{The system of ionized gas dynamics}
\author{Fumioki ASAKURA
\footnote{Research Center for Physics and Mathematics, Osaka Electro-Communication University, Neyagawa, Osaka, Japan,
 e-mail: {\tt\small asakura@osakac.ac.jp}}
\hspace{4ex}
Andrea CORLI
\footnote{Department of Mathematics and Computer Science, University of Ferrara, Via Machiavelli 30, 40121 Ferrara, Italy,
e-mail: {\tt\small andrea.corli@unife.it}}
}
\maketitle
\begin{abstract}
The aim of this paper is to study a system of three equations for ionized gas dynamics at high temperature, in one spatial dimension. In addition to the mass density, pressure and particle velocity, a further quantity is needed, namely, the degree of ionization. The system is supplemented by the first and second law of thermodynamics and by an equation of state; all of them involve the degree of ionization. At last, under the assumption of thermal equilibrium, the system is closed by requiring Saha's ionization equation.

The geometric properties of the system are rather complicated: in particular, we prove the loss of convexity (genuine nonlinearity) for both forward and backward characteristic fields, and hence the loss of concavity of the physical entropy. This takes place in a small bounded region, which we are able to characterize by  numerical estimates on the state functions. The structure of shock waves is also studied by a detailed analysis of the Hugoniot locus, which will be used in a forthcoming paper to study the shock tube problem. 
\end{abstract}

\smallskip

\textit{\quad 2010~Mathematics Subject Classification:} 35L65, 35L67, 76N15.

\smallskip

\textit{\quad Key words and phrases:}
Systems of conservation laws, ionized gas, Hugoniot locus.

\section{Introduction}\label{sec:Introduction}

The thermodynamical variables temperature, specific entropy, specific internal energy, pressure, specific volume and the velocity of the gas are denoted in this paper by $T$, $S$, $e$, $p$, $v$ and $u$, respectively. For brevity we simply refer to $S$ and $e$ as the entropy and internal energy, respectively. In the simplest (non-ionized) thermodynamical system, an {\em equation of state} relating the variables $p$, $v$ and $T$, for example, is assigned. Then, the knowledge of any two of them allows to fully determine the state of the system, since the other variables are determined by the first and second law of thermodynamics
\begin{equation}\label{eq:second-principle}
T\d S = \d e + p\d v.
\end{equation}
However, if the gas is partly ionized as in this paper, further details must be taken into account. Ionization is a process where an atom or a molecule becomes charged due to the loss or gain of electrons. When a gas is heated to a high temperature, at first almost all of its molecules dissociate and it behaves like a monatomic gas. However, if the gas is heated to an even higher temperature, of the order of thousands of degrees Kelvin, some of its atoms become ionized according to the first ionization reaction ${\rm X} \rightarrow {\rm X}^+ + {\rm e}^-$, where $\rm{X}$ is an atom of the monatomic gas, ${\rm X}^+$ an ion and ${\rm e}^-$ an electron. In the case of hydrogen, at this level the gas is made up by atoms, protons and electrons; in heavier atoms the first ionization can be followed by a second ionization and so on. For simplicity we focus on the case of a gas that undergoes only one ionization and denote the concentration (number per unit volume) of atoms, ions and electrons by $n_{\rm a}$, $n_{\rm i}$ and $n_{\rm e}$, respectively.

\par
At any given (high) temperature $T,$ this ionization reaction reaches a state of thermal equilibrium  analogous to the chemical equilibrium for ordinary chemical reaction. 
We may assume that the ratio 
\[
\frac{n_{\rm i}n_{\rm e}}{n_{\rm a}} 
\]
depends only on the temperature $T$; an actual formula was derived   in 1920 by M. Saha \cite{Saha}, see also \cite{Bradt,Fermi,Vincenti-Kruger,Zeldovich-Razier}. The state of an ionized gas depends on a further variable which measures the level of ionization, namely, the {\em degree of ionization} 
\[
\alpha = \dfrac{n_{\rm e}}{n_{\rm a} + n_{\rm i}}.
\] 
However, by means of Saha's formula, the degree of ionization 
can be expressed as a function of the density $\rho$ and $T$, and this makes consistent the formulation of the model. Of course, the expressions of the thermodynamical variables depend on the degree of ionization; for instance, in the case of a monatomic gas, the pressure of the ionized gas can be written as 
\[
p = (1+\alpha)\frac{R}{m}\rho T,
\]
where $R$ is the universal gas constant and $m$ the molecular mass of the gas. In the non-ionized case $\alpha=0$ we recover the usual expression for the pressure; if $\alpha\ne0$ the gas is not ideal \cite{Smoller}.

The aim of this paper is to study the mathematical properties of the gasdynamics system 
\begin{equation}\label{eq:system}
\left\{
\begin{array}{l}
\rho_t + (\rho u)_x = 0,
\\
\ds(\rho u)_t + (\rho u^2 + p)_x = 0,\rule{0ex}{2.75ex}
\\
\ds\left(\rho E\right)_t + \left(\rho u E + pu\right)_x = 0,\rule{0ex}{2.75ex}
\end{array}
\right.
\end{equation}
for an ionized gas; here, $E = \frac{u^2}{2} + e$ is the (specific) total energy. To the best of our knowledge such a study has never been done previously in spite of the fact that the system of gasdynamics attracted the interest of several researchers in the last decade, however mostly for ideal gases \cite{Smoller}; we quote \cite{Menikoff-Plohr} for the case of real gases. The main motivation for our study was provided by the paper \cite{Fukuda-Okasaka-Fujimoto} \footnote{An English translation of \cite{Fukuda-Okasaka-Fujimoto} is available upon request to F. Asakura.}, where the authors made laboratory experiments about the reflection at an interface of a shock wave in an ionized gas. In this paper we assume that the gas is {\em monatomic}, differently from \cite{Fukuda-Okasaka-Fujimoto}, where a two-component gas is considered. This choice simplifies the analysis of the system while catching however the main features of ionization; more general cases (non-monatomic or multicomponent gases) can be considered as well but at the price of much heavier computations.

After a short review of basic thermodynamics in Section \ref{s:basic-thermodynamics} and a first study of the thermodynamic variables of the model in Section \ref{s:thermo-variables},  we show in Section \ref{s:properties} that system \eqref{eq:system} is strictly hyperbolic, as in the non-ionized case; however, differently from the latter case, the genuine non-linearity of the extreme eigenvalues of the system is lost and we study the set where this happens. We refer to \cite{Dafermos, Smoller} for more information on systems of conservation laws. Our analysis is completed by a detailed investigation of the Hugoniot locus of the system in Section \ref{sec:Hugoniot} and by a study of the entropy increase in Section \ref{sec:entropy}. Moreover, a study of the integral curves in provided in Section \ref{sec:rarefaction curves}. At last, in Section \ref{sec:HTLM} we study a related model, which is deduced by an approximation of Saha's law at very high temperatures; in this case, the system does not loose the genuine nonlinearity property. 

In a second forthcoming paper we exploit the results obtained in this paper to study the shock reflection in an electromagnetic shock tube (T-tube) with a reflector. We thus provide a rigorous mathematical basis to the physical phenomena observed in \cite{Fukuda-Okasaka-Fujimoto}.

\section{Basic Thermodynamics}\label{s:basic-thermodynamics}
\setcounter{equation}{0}

The equation of state of an {\em ideal} gas can be written as 
\cite[(6)]{Fermi}
\begin{equation}\label{eq:p_av}
p = \frac{\rho}{m_{\rm p}}\,kT = n_{\rm p}kT,
\end{equation}
where $m_{\rm p}$ is the particle mass, $n_{\rm p}$ the total number of particles per unit volume and $k$ is Boltzmann's constant; we clearly have $\rho = m_{\rm p}n_{\rm p}$. The molecular mass is defined as
\begin{equation}\label{eq:molecular-mass}
m = N_0 m_{\rm p},
\end{equation}
where $N_0$ is Avogadro number. If $R = k N_0$ is the universal gas constant, then \eqref{eq:p} becomes 
\cite[(8)]{Fermi}
\begin{equation}\label{eq:p}
p = \frac{R}{m}\rho T,
\end{equation}
or $p = M R\rho T$ by introducing the number $M=\frac{1}{m}$ of moles. We shall often use $p$ and $T$ as independent thermodynamical variables; in order to compute the other variables we introduce the (specific) enthalpy $H= e+pv$. Using the enthalpy equation \eqref{eq:second-principle} becomes
\begin{equation}\label{eq:enthalpy}
   dH = TdS + vdp.
\end{equation}
We also introduce the Gibbs function  
\(
G = H - TS
\), see \cite[(111)]{Fermi}.
By \eqref{eq:second-principle} we deduce $dG = v\, dp - S\,dT$, whence
\begin{equation}\label{eq:Gibbs-derivatives}
\left(\frac{\partial G}{\partial p}\right)_T = v, \qquad \left(\frac{\partial G}{\partial T}\right)_p = -S.
\end{equation}
As usual, a subscript as $T$ or $p$ above means that the derivative is computed when that variable is fixed. By \eqref{eq:Gibbs-derivatives} we deduce the compatibility condition
\begin{equation}\label{eq:Maxwell}
\left(\frac{\partial v}{\partial T}\right)_p = -\left(\frac{\partial S}{\partial p}\right)_T,
\end{equation}
which is one of the so-called Maxwell reciprocity relations. 
In turn, by \eqref{eq:Gibbs-derivatives} and \eqref{eq:Maxwell} we obtain
\begin{eqnarray}
\left(\frac{\partial H}{\partial p} \right)_{T}
  &=& T\left(\frac{\partial S}{\partial p} \right)_{T} + v
  \ =\  -T\left(\frac{\partial v}{\partial T} \right)_{p} + v,\label{eq:dH/dp}
  \\
  \left(\frac{\partial H}{\partial T} \right)_{p}
  &=& T\left(\frac{\partial S}{\partial T}\right)_{p}.\label{eq:dH/dT}
\end{eqnarray}

\begin{example}[Polytropic gas]
In the case of a polytropic gas, i.e., an ideal gas whose specific heat is constant, we have \cite[(29)]{Fermi}
\begin{equation}\label{eq:energy}
e=C_vT,
\end{equation}
where $C_v$ is the specific heat at constant volume \cite[(34)]{Fermi}. 
We also define the specific heat at constant pressure $C_p = \frac{R}{m} + C_v$. Then, $ e= C_vT$ and $H= C_pT$. The adiabatic constant is
$\gamma = \frac{C_p}{C_v} = 1 + \frac{R}{mC_v}$ \cite[(37)]{Fermi}. 
We conclude by computing the entropy $S$. By \eqref{eq:second-principle} we deduce
\[
\left(\frac{\partial S}{\partial e}\right)_v = \frac{1}{T},\qquad \left(\frac{\partial S}{\partial v}\right)_e = \frac{p}{T},
\]
and then $S = C_p\log T - \frac{R}{m}\log p + \hbox{Const}$ by \eqref{eq:p} and \eqref{eq:energy}. By introducing the non-dimensional entropy $\eta = \frac{m}{R}S$ we easily find
\begin{equation}\label{eq:pvT}
p = a^{2\gamma}e^{(\gamma-1)\eta}v^{-\gamma},
\qquad
v = a^2p^{-\frac{1}{\gamma}}e^{\frac{\gamma-1}{\gamma}\eta},\qquad
   T = a^2p^{\frac{\gamma-1}{\gamma}}e^{\frac{\gamma-1}{\gamma}\eta}
\end{equation}
where
\begin{equation}\label{e:a}
a^2=\frac{R}{m}.
\end{equation}
In the case of a monatomic gas, we have $C_v=\frac32a^2$, $C_p = \frac52 a^2$ and $\gamma = \frac53$ \cite[(34) and (37)]{Fermi}. Then,
\begin{equation}\label{eq:eHS}
e = \frac32a^2T, \qquad H= \frac52a^2T, \qquad S = \frac52a^2\log T - a^2\log p + \hbox{Const}.
\end{equation}
\end{example}

Now, we briefly discuss ionization; to avoid multiple references we mainly refer to \cite{Bradt} but analogous expressions can be easily found in \cite{Fermi, Vincenti-Kruger, Zeldovich-Razier}. Saha's equation provides the degree of ionization of an atomic element as a function of temperature and electron density, under assumptions of thermal equilibrium. In a general form it can be written \cite[(4.15)]{Bradt}
\begin{equation}\label{eq:saha-B}
  \frac{n_{r+1}n_{\rm e}}{n_r} = \frac{G_{r+1}g_{\rm e}}{G_r} \frac{(2\pi m_{\rm e} kT)^{\frac{3}{2}}}{h^3}\, e^{-\frac{\chi_r}{kT}}.
\end{equation}
Here $n_r$ and $n_{r+1}$ are the number densities of atoms in the ionization state $r$ (where $r$ electrons are missing) and the ionization state $r+1$ (where $r+1$ electrons are missing) of a given element, $n_{\rm e}$ is the electron number density, $G_r$ and $G_{r+1}$ are the partition functions of the two states, $g_{\rm e} =2$ is the statistical weight of the electron, $m_{\rm e}$ is the electron mass, $h$ is the Planck constant, $k$ is the Boltzmann constant, $\chi_r$ is the ionization potential from state $r$ to state $r+1$. We also denote  the degree of ionization by
\begin{equation}\label{eq:degree-ionization}
\alpha = \dfrac{n_{\rm e}}{n_{\rm a} + n_{\rm i}}.
\end{equation}
\par
As we mentioned in the Introduction, we consider the case $r=0$ of a single ionization from a neutral state of a monatomic gas; 
then \eqref{eq:saha-B} becomes 
\begin{equation}\label{eq:saha}
\frac{n_{\rm i} n_{\rm e}}{n_{\rm a}} = \frac{2Z_{\rm i}}{Z_{\rm a}} \frac{(2\pi m_{\rm e} kT)^{\frac{3}{2}}}{h^3}\, e^{-\frac{T_{\rm i}}{T}},
\end{equation}
where we denoted by $Z_{\rm a}=G_0$ and $Z_{\rm i}=G_1$ the partition functions of the neutral state and $1$-ionized state, respectively; $T_{\rm i}=\frac{\chi_0}{k}$ is the dissociation energy expressed by the temperature.  
Moreover, $n_{\rm e} = n_{\rm i}$ and then $\alpha=\frac{n_{\rm i}}{n_{\rm a}+n_{\rm i}}$. The total number of atoms and ions is $n_{\rm p} = n_{\rm a}+n_{\rm i}$ and we can assume that $n_{\rm p} = \rho/m_{\rm p}$. As a consequence, by \eqref{eq:p_av} we deduce the pressure law for an ionized gas,
\begin{equation}\label{eq:pressure-ion-mp}
p = (n_{\rm a}+n_{\rm i}+n_{\rm e})kT = (1+\alpha)n_{\rm p}kT = (1+\alpha)\frac{\rho}{m_{\rm p}}kT,
\end{equation}
or better, by using the molecular mass \eqref{eq:molecular-mass},
\begin{equation}\label{eq:pressure-ion}
p = (1+\alpha)\frac{R}{m}\rho T.
\end{equation}
The equation of state \eqref{eq:pressure-ion} expresses $\alpha$ as a function of $p$, $\rho$ and $T$. Indeed, as we show in the next Lemma \ref{lem:alpha}, $\alpha$ does {\em not} depend on $\rho$ and $p$ independently but solely on one of these two variables. As a consequence, an ionized gas is not an ideal gas, which is characterized instead by the equation of state \eqref{eq:p}. The specific volume $v$ is deduced as
\begin{equation}\label{eq:spec-vol-ion}
v = (1+\alpha)\frac{RT}{mp}.
\end{equation}
In the current case of an ionized gas, the specific internal energy should include not only kinetic, but also ionization energy; then
$
e = \frac32\frac{p}{\rho} + \frac{\chi_0}{m_p}\alpha.
$
But
$
\frac{\chi_0}{m_p}= \frac{k T_{\rm i}N_0}{m} = \frac{R}{m}T_{\rm i}
$
and then, since the gas is monatomic,
\begin{equation}\label{eq:e}
e  =  \frac32 \frac{R}{m}(1+\alpha)T + \frac{RT_{\rm i}}{m}\alpha,
\qquad
H  =  \frac52 \frac{R}{m}(1+\alpha)T + \frac{RT_{\rm i}}{m}\alpha. 
\end{equation}
We stress that, differently from ideal gases where $e=e(T)$, for an ionized gas the specific internal energy $e$ depends on {\em both} $p$ and $T$. The above expressions \eqref{eq:e} slightly change for non-monatomic gases; we emphasize that most numerical coefficients that occur in the following (as $5/2$ or $5/3$) are merely due to this assumption. 


\section{The Thermodynamic Variables of an Ionized Gas}\label{s:thermo-variables}
\setcounter{equation}{0}

In this section we provide several explicit expressions for the thermodynamical variables of an ionized gas. In the proofs we use from time to time the shorthand notation
\begin{equation}\label{eq:q}
\tau = \frac{\Ti}{T} \quad \hbox{ and } \quad
q=\frac52 +\tau.
\end{equation}

First, we study the properties of the degree $\alpha$ of ionization. By exploiting the notation introduced in the previous section, we denote
\[
\kappa = \frac{h^3}{(2\pi m_{\rm e})^{\frac{3}{2}}k^{\frac{5}{2}}} \frac{Z_{\rm a}}{2 Z_{\rm i}} \quad \hbox{ and } \quad   \bar{\kappa} = m_{\rm p}\frac{2Z_{\rm i}}{Z_{\rm a}} \frac{(2\pi m_{\rm e} k)^{\frac{3}{2}}}{h^3}.
\]
We notice that $\bar\kappa = \frac{m_{\rm p}}{k}\frac{1}{\kappa} = \frac{m}{R}\frac{1}{\kappa}$.
\begin{lemma}[Degree of ionization]\label{lem:alpha}
The degree $\alpha$ of ionization can be expressed as a function of $\rho$ and $T$ through the formula
\begin{equation}\label{eq:deg-ionization-rhoT}
    \frac{\alpha^2}{1-\alpha} = \frac{\bar{\kappa}}{\rho}T^{\frac32}e^{-\frac{\Ti }{T}}.
\end{equation}
The degree of ionization can also be expressed as a function of $p$ and $T$ by the formula
\begin{equation}\label{eq:deg-ionization}
    \alpha = \left(1 + \kappa pT^{-\frac{5}{2}}e^{\frac{\Ti }{T}}\right)^{-\frac{1}{2}}.
\end{equation}
Moreover, we have
\begin{equation}
\left(\frac{\partial \alpha}{\partial p} \right)_{T}
 =  -\frac{1}{2p}\,\alpha(1 - \alpha^2)
 \quad\hbox{ and }\quad
\left(\frac{\partial \alpha}{\partial T} \right)_{p}
 =  \frac{1}{2T}\,\alpha(1 - \alpha^2)\qq.
\label{eq:alpha-derivatives}
\end{equation}
At last,
\begin{eqnarray}
-\frac{T}{p}\left(\frac{\partial \alpha}{\partial T} \right)_{p} & = & \qq\left(\frac{\partial \alpha}{\partial p} \right)_{T}.
\label{eq:alpha-pTderivative}
\end{eqnarray}
\end{lemma}
\begin{proof}
We first notice that
\[
\frac{n_{\rm i} n_{\rm e}}{n_{\rm a}} = \frac{n_{\rm i}^2}{n_{\rm a}} = n_{\rm p}\frac{\alpha^2}{1-\alpha}.
\]
Therefore, \eqref{eq:saha} can be written as (see  \cite[(209)]{Fermi}, \cite[\S V.4, (4.9)]{Vincenti-Kruger}) 
\begin{equation*}\label{eq:saha1}
\frac{\alpha^2}{1-\alpha}
 = \frac{m_{\rm p}}{\rho}\frac{2Z_{\rm i}}{Z_{\rm a}} \frac{(2\pi m_{\rm e} kT)^{\frac{3}{2}}}{h^3}\, e^{-\frac{T_{\rm i}}{T}},
\end{equation*}
which proves \eqref{eq:deg-ionization-rhoT}. By \eqref{eq:pressure-ion-mp}, we deduce \cite[\S5, (4.11)]{Vincenti-Kruger}
\begin{equation}\label{eq:saha2}
\frac{\alpha^2}{1-\alpha^2}
= \frac{2Z_{\rm i}}{Z_{\rm a}} \frac{(2\pi m_{\rm e})^{\frac{3}{2}}(kT)^\frac52}{ph^3}\, e^{-\frac{T_{\rm i}}{T}},
\end{equation}
whence \eqref{eq:deg-ionization}.
To compute \eqref{eq:alpha-derivatives} it is convenient to introduce the notation
$\beta(p,T) = \kappa p T^{-\frac52}e^{\frac{T_{\rm i}}{T}}$; by \eqref{eq:saha2} we have
\begin{equation}\label{eq:beta-alpha}
\alpha = (1+\beta)^{-\frac12}, \qquad \beta = \frac{1-\alpha^2}{\alpha^2}.
\end{equation}
Now, we compute
\[
\left(\frac{\partial \alpha}{\partial p} \right)_{T}
= -\frac{\alpha^3}{2} \kappa T^{-\frac52}e^{\frac{T_{\rm i}}{T}}
 = -\frac{\alpha^3}{2p} \beta,
\qquad  \left(\frac{\partial \alpha}{\partial T} \right)_{p}
= -\frac{\alpha^3}{2} \left(\frac{\partial \beta}{\partial T} \right)_{p}
  = \frac{\alpha^3}{2}\frac{\beta}{T}\left(\frac{5}{2} + \frac{\Ti }{T}\right),
\]
whence \eqref{eq:alpha-derivatives}. Identity \eqref{eq:alpha-pTderivative} directly follows by \eqref{eq:alpha-derivatives}. However, it can also be obtained without computing \eqref{eq:alpha-derivatives}: indeed, by \eqref{eq:second-principle} and \eqref{eq:Maxwell} we deduce
\[
\left(\frac{\partial H}{\partial p}\right)_T = T\left(\frac{\partial S}{\partial p}\right)_T + v = -T\left(\frac{\partial v}{\partial T}\right)_p + v.
\]
By the above formula, differentiating \eqref{eq:e} and using \eqref{eq:spec-vol-ion} we deduce \eqref{eq:alpha-pTderivative}.
\end{proof}
\begin{nb}[Physical values]\label{rem:values}
We report here some physical values which are commonly used in the figures below. As we wrote in the Introduction, we focus on the simple case of a hydrogen gas. In this case the specific gas constant $a^2$ is $8314 {\rm\, J\, kg}{}^{-1} {\rm K}{}^{-1}$ and $\kappa = 29.9774$. Moreover $\Ti = 1.578 \times 10^5{\rm K}$.  We frequently use the following values, where $\alpha$ is computed through \eqref{eq:deg-ionization},  see \cite{Fukuda-Okasaka-Fujimoto}:
\[
\begin{array}{lll}
T = 750{\rm K}, & p=1466.3 {\rm Pa}\ (11 {\rm Torr}), & \alpha = 3.8418 \times 10^{-45},
\\
T = 300{\rm K}, & p=1466.3 {\rm Pa}\ (11 {\rm Torr}), & \alpha = 3.5929 \times 10^{-114}.
\end{array} 
\]
\end{nb}

\begin{nb}The equation of state \eqref{eq:pressure-ion} allows to recover easily the equation of state of the non-ionized case $\alpha=0$. Lemma \ref{lem:alpha} shows that this is not the case for other thermodynamics variables: if we let $\alpha=0$ in \eqref{eq:deg-ionization-rhoT} or \eqref{eq:deg-ionization}, then we formally deduce $T=0$ (or $T=\infty$, or $\rho=\infty$, even less physically meaningful). The reason depends on Saha's formula \eqref{eq:saha}, which allows ionization at any positive temperature. This shows that the non-ionized case cannot be deduced as the limit for $\alpha\to0$ of the ionized one.  
\end{nb}

\begin{figure}[htbp]	
  \includegraphics[width=8cm]{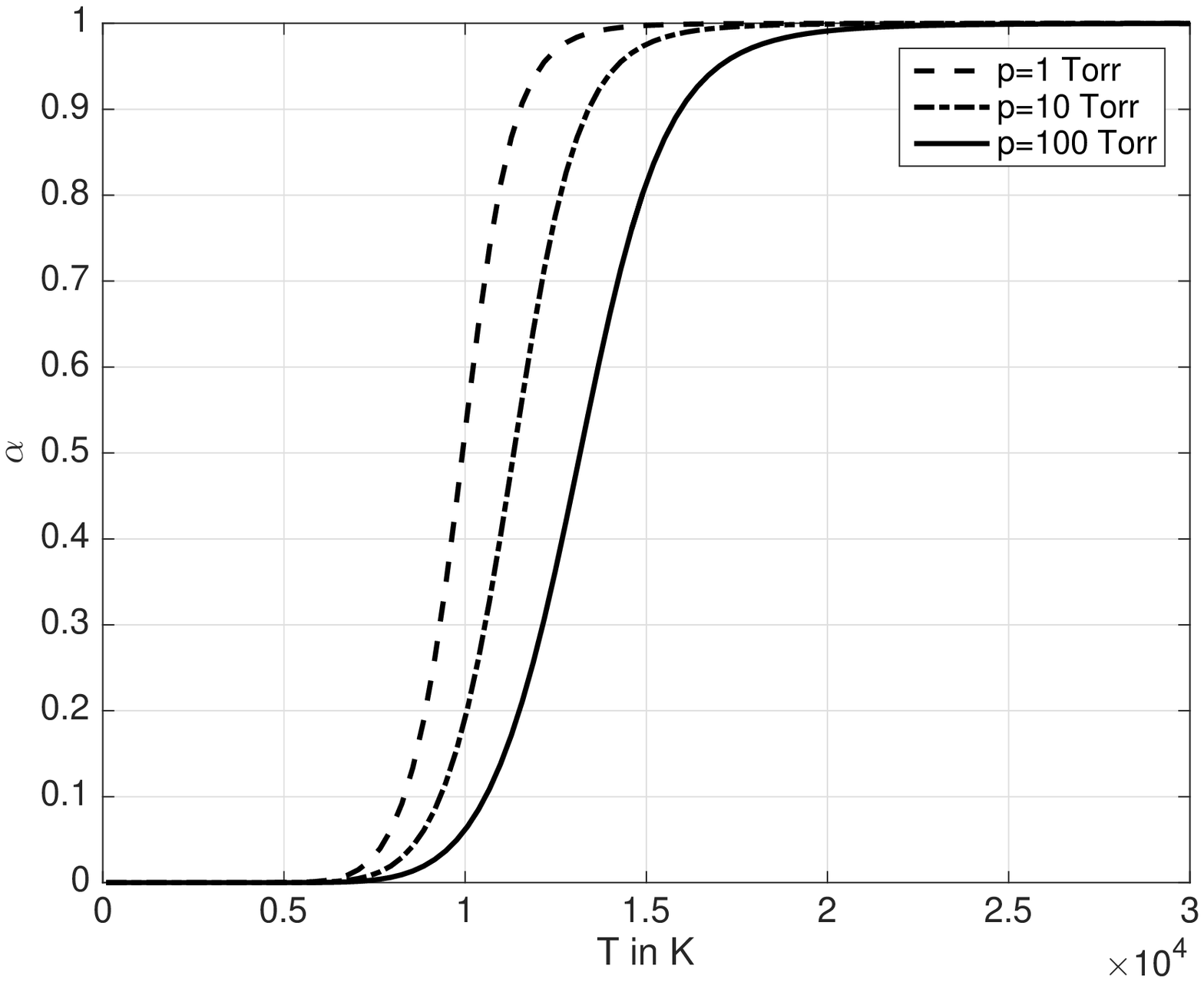}
  \includegraphics[width=8cm]{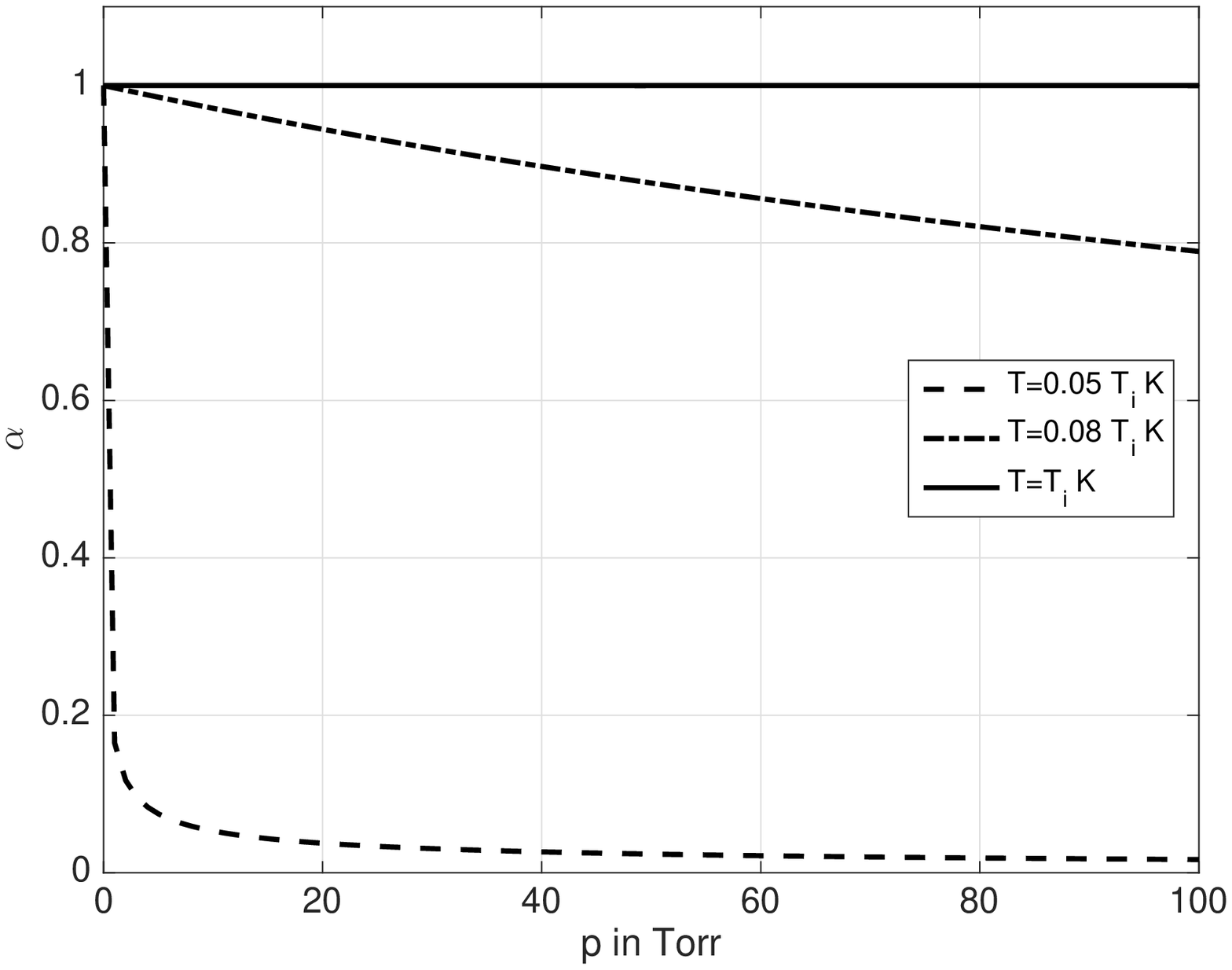}
\caption{The ionization degree $\alpha$ given by \eqref{eq:deg-ionization}: as a function of the temperature $T$, left, and of the pressure, right. We refer to Remark \ref{rem:values} for the numerical values used here and in the following.}
  \label{fig:alphaT}
\end{figure}
As a consequence of Lemma \ref{lem:alpha}, the pressure law \eqref{eq:pressure-ion} can be written more carefully as
\begin{equation}\label{eq:prhoT}
p = p(\rho, T) = \left(1+\alpha(\rho,T)\right)\frac{R}{m}\rho T,
\end{equation}
or as
\begin{equation}\label{eq:p-alphaT}
    p = p(\alpha, T) = \frac1\kappa\,\frac{1 - \alpha^2}{\alpha^2}T^{\frac{5}{2}}e^{-\frac{\Ti }{T}}.
\end{equation}

\begin{lemma}[Pressure]\label{lem:p_rho_T}
We have 
\begin{equation}
p_\rho(\rho, T)  = \frac{2}{2-\alpha} \frac{R}{m}T >0 \quad\hbox{ and }\quad
p_T(\rho, T)  = \frac{2}{2-\alpha} \left[1+\frac{1}{2}\alpha(1-\alpha)\qq \right]\frac{R}{m}\rho>0.
 \label{eq:p_rhoT}
\end{equation}
\end{lemma}
\begin{proof}
Instead of differentiating directly \eqref{eq:prhoT}, by Lemma \ref{lem:alpha} we understand $\alpha$ as a function of $p$ and $T$ in that expression. We denote $F(\rho,T,p) = \left(1+\alpha(p,T)\right)\frac{R}{m}\rho T - p$; by $\eqref{eq:alpha-derivatives}_1$ we see that
\[
F_p = \left(\frac{\partial \alpha}{\partial p}\right)_T\frac{R}{m}\rho T -1 
= -\frac12(2-\alpha)(1+\alpha) <0.
\]
By the Implicit Function Theorem and $\eqref{eq:alpha-derivatives}_2$ we deduce
\[
F_\rho = \left(1+\alpha\right)\frac{R}{m}T,
\qquad
F_T = 
\frac{R}{m}(1+\alpha)\rho\left[1+\alpha(1-\alpha)\frac{q}{2} \right].
\]
Since $p_\rho = -\frac{F_\rho}{F_p}$ and $p_T  = -\frac{F_T}{F_p}$, we obtain \eqref{eq:p_rhoT}.
\end{proof}


\begin{lemma}[Internal energy]\label{lem:int-energy}
If we express the internal energy by means of $\rho$ and $T$ we have
\[
e_T(\rho,T)>0.
\]
\end{lemma}
\begin{proof}
By Lemma \ref{lem:p_rho_T} we may express $\alpha$ as a function of $\rho$ and $T$; by \eqref{eq:alpha-derivatives} and $\eqref{eq:p_rhoT}_2$ we compute
\begin{align}\label{eq:alpha_rho}
\frac{\partial }{\partial T}\alpha\left(p(\rho,T),T\right) & = \alpha_p p_T + \alpha_T  = \frac{\alpha(1-\alpha)}{(2-\alpha)T}(q-1).
\end{align}
Then by \eqref{eq:e} and \eqref{eq:alpha_rho} we deduce
\begin{align*}
e_T(\rho,T) & = \frac{R}{m}\left[\frac32(1+\alpha)T + \Ti \alpha\right]_T
 = \frac{R}{m}\frac{\alpha(1-\alpha)}{(2-\alpha)}\left[q^2 + \frac{q}{2}+\frac{3}{\alpha(1-\alpha)} \right],
\end{align*}
whence $e_T>0$. For an alternative proof we can directly use expression  \eqref{eq:deg-ionization-rhoT}.
\end{proof}


\begin{proposition}[Entropy]\label{prop:entropy}
The dimensionless entropy $\eta=\frac{m}{R}S$ can be expressed as a function of $p$ and $T$ through the formula
\begin{equation}\label{eq:eta}
  \eta(p,T) = -\log p + 2 \tanh^{-1} \alpha + \qq\alpha
  + \frac{5}{2} \log T + C,
\end{equation}
where $C$ is a constant; we have
\begin{equation}
\fr{\eta}{p}{T} \le -\frac{1}{p}<0
\quad \hbox{ and }\quad 
\left(\frac{\partial \eta}{\partial T} \right)_{p} \ge \frac{5}{2T}(1+\alpha) > 0.
\label{eq:deta/dpT}
\end{equation}
Moreover, the dimensionless entropy $\eta$ can be expressed as a function of $\alpha$ and $T$ by the formula
\begin{equation}\label{eq:eta-alphaT}
    \eta(\alpha,T) =  - 2\log \frac{1 - \alpha}{\alpha} +  (1 + \alpha)\qq  + C
\end{equation}
and we have
\begin{equation}
\fr{\eta}{\alpha}{T} = \frac{2}{\alpha(1-\alpha)} + \qq >0
\quad\hbox{ and }\quad
\left(\frac{\partial \eta}{\partial T} \right)_{\alpha} = -(1+\alpha)\frac{T_{\rm i}}{T^2} < 0.
\label{eq:deta/dalphaT}
\end{equation}
\end{proposition}
\begin{proof}
By Maxwell relation \eqref{eq:Maxwell}, \eqref{eq:alpha-pTderivative} and \eqref{eq:dH/dT} we find
\begin{align}
  \left(\frac{\partial S}{\partial p} \right)_{T}
  &= -\left(\frac{\partial v}{\partial T} \right)_{p} \ = \ -\frac{R}{m}\left[\frac{1}{p} (1 + \alpha) + \frac{T}{p}\left(\frac{\partial \alpha}{\partial T} \right)_{p}\right]
  \label{eq:dS/dp-first}
  \\
  &= \frac{R}{m}\left[- \frac{1}{p} (1 + \alpha) + q\left(\frac{\partial \alpha}{\partial p} \right)_{T}\right],
  \label{eq:dS/dp}
  \\
  \left(\frac{\partial S}{\partial T} \right)_{p}
  &= \frac{1}{T}\left(\frac{\partial H}{\partial T} \right)_{p}
  = \frac{R}{m}\left[\frac{5}{2T}(1+\alpha) + q\left(\frac{\partial \alpha}{\partial T} \right)_{p}\right].
  \label{eq:dS/dT}
\end{align}
By integrating \eqref{eq:dS/dp} with respect to $p$ we obtain
\begin{eqnarray}
\eta & = & - \log p -\int\frac{\alpha}{p}\d p + q\alpha + f(T),
\label{eq:etafirst}
\end{eqnarray}
where $f(T)$ is an arbitrary function of $T$. To compute the integral we introduce the notation $r(T) = \frac{\beta(p,T)}{p}$, see \eqref{eq:saha2} and \eqref{eq:beta-alpha}.
By the change of variables $1+r(T)p = y^2$ we deduce
\[
\int\frac{\alpha}{p}\d p  = \int\frac{1}{p\sqrt{1+r(T)p}}\d p = 2\int\frac{1}{y^2-1}\d y = \log\frac{y-1}{y+1} = \log\frac{1-\alpha}{1+\alpha}
= - 2 \tanh^{-1} \alpha
\]
and then \eqref{eq:etafirst} becomes $\eta  =  - \log p + 2 \tanh^{-1} \alpha + q\alpha + f(T)$. To compute $f$, we first differentiate \eqref{eq:etafirst} with respect to $T$ and use \eqref{eq:alpha-pTderivative}, \eqref{eq:q} to deduce
\begin{equation}
\left(\frac{\partial \eta}{\partial T}\right)_p  =  -\int \frac{1}{p}\left(\frac{\partial \alpha}{\partial T}\right)_p \d p - \frac{T_{\rm i}}{T^2}\alpha + q\left(\frac{\partial \alpha}{\partial T}\right)_p + f'(T)
= \frac{5\alpha}{2T} + q\left(\frac{\partial \alpha}{\partial T}\right)_p + f'(T).
\label{eq:deta/dT}
\end{equation}
Now, we compare \eqref{eq:deta/dT} with \eqref{eq:dS/dT} and find that $f'(T) = \frac52\frac{1}{T}$, so that
$f(T) = \frac52\log T$. Then \eqref{eq:eta} follows. By \eqref{eq:dS/dp-first}, $\eqref{eq:alpha-derivatives}_2$ and by \eqref{eq:dS/dT}, \eqref{eq:alpha-pTderivative} we deduce \eqref{eq:deta/dpT}. To prove \eqref{eq:eta-alphaT}, by \eqref{eq:eta} and \eqref{eq:p-alphaT} we compute
\begin{eqnarray*}
  \eta &=& - \log p + \log \frac{1 + \alpha}{1 - \alpha} + q\alpha + \frac{5}{2}\log T + C 
     = - \log \left(\frac{1 - \alpha}{\alpha}\right)^2  + \tau +  q\alpha  + C,
\end{eqnarray*}
whence \eqref{eq:eta-alphaT}. Then, \eqref{eq:deta/dalphaT} follows.
\end{proof}


Notice that if $\alpha=0$, by \eqref{eq:eta} we recover the entropy \eqref{eq:eHS}. 


\begin{proposition}[Temperature]
The temperature $T$ can be globally expressed as a function of $p$ and $\eta$.

Moreover, the level curves $\eta(p,T) = c$ lying on the plane $(p,T)$ are the graphs of functions $T_c(p)$ of $p$ satisfying
\begin{equation}\label{eq:dT/dp}
\frac{\partial T_c}{\partial p} = \frac{T_c}{p}
\frac{1+\frac12 \alpha(1-\alpha)\left(\frac52 + \frac{T_{\rm i}}{T}\right)}{\frac52 + \frac12 \alpha(1-\alpha)\left(\frac52 + \frac{T_{\rm i}}{T}\right)^2}.
\end{equation}
\end{proposition}
\begin{proof}
On the level curves $\eta(p,T) = c$ we have
$\!\d\eta = \left(\frac{\partial \eta}{\partial p}\right)_T\!\d p + \left(\frac{\partial \eta}{\partial T}\right)_p\!\d T=0$; then, on these curves, we have
\[
\frac{\partial T}{\partial p} = -\frac{\fr{\eta}{p}{T}}{\fr{\eta}{T}{p}} =: F(p,T),
\]
where we omitted for simplicity the subscript $c$. We notice that $F$ makes sense and $F(p,T)>0$ because of \eqref{eq:deta/dpT}. Moreover, by \eqref{eq:dS/dp-first}, \eqref{eq:dS/dT} and $\eqref{eq:alpha-derivatives}_2$ we deduce
\[
\fr{\eta}{p}{T}  = -\frac{1+\alpha}{p}\left(1+\frac{\alpha}{2}(1-\alpha)\left(\frac52 + \frac{T_{\rm i}}{T}\right)\right),
\qquad
\fr{\eta}{T}{p}  = \frac{1+\alpha}{T}\left(\frac52 + \frac{\alpha}{2}(1-\alpha)\left(\frac52 + \frac{T_{\rm i}}{T}\right)^2   \right),
\]
and then \eqref{eq:dT/dp} follows at least locally and defines a strictly increasing function $T(p)$. Indeed, we claim that $T$ is defined for every $p>0$. To prove this claim, we first remark that $\frac{\partial T}{\partial p} \sim \frac25\frac{T}{p}$ both for $T\to0+$ and for $T\to+\infty$, because in those cases $\alpha\to0$ and $\alpha \to 1$, respectively. This means that $T\sim C p^\frac25$ for $T\to0+$ and $T\to+\infty$, for $C$ an arbitrary positive constant, which proves our claim. The claim also implies that for every fixed pair $(p,\eta)$ we uniquely find the relative value of $T$. This proves the proposition.
\end{proof}


\section{Properties of the System of Ionized Gasdynamics}\label{s:properties}
\setcounter{equation}{0}
In this section we study the main properties of system \eqref{eq:system}, where $p$ and $e$ are defined by \eqref{eq:pressure-ion}, \eqref{eq:e}, respectively, with $\alpha$ as in \eqref{eq:deg-ionization}.

\subsection{Eigenvalues and eigenvectors} 

System \eqref{eq:system} can be written in Lagrangian coordinates as \cite{Smoller}
\begin{equation}\label{eq:Lagrangian}
\left\{
  \begin{array}{l}
   v_t - u_\xi = 0,\\
   u_t + p_\xi = 0,\\
   \left(e + \frac{1}{2}u^2\right)_t + (pu)_\xi = 0.
  \end{array}
\right.
\end{equation}
For $C^1$ solutions, equation $\eqref{eq:Lagrangian}_3$ can be written as $S_{t} = 0$
and we obtain
\begin{equation}\label{eq:Lagrangian-n}
\left\{
  \begin{array}{r}
   v_t - u_{\xi} = 0,\\
   u_t + p_{\xi} = 0,\\
   S_t = 0.
  \end{array}
\right.
\end{equation}
We notice that equation $\eqref{eq:Lagrangian-n}_3$ is equivalent in Eulerian coordinates to
\begin{equation}\label{eq:entropy conservation}
\left(\rho S\right)_t + \left(\rho uS\right)_x = 0.
\end{equation}
In the variables $(p,u,S)$ we have $v_t=v_pp_t$ and equation $\eqref{eq:Lagrangian-n}_1$ becomes $p_t-\frac{1}{v_p}u_\xi=0$. In these variables, the {\it Lagrangian} characteristic speeds are $\lambda_{\pm} = \pm \frac{1}{\sqrt{-v_p}}$ and  $\lambda_0 = 0,$ with corresponding characteristic vectors $r_{\pm} = \left(\pm 1,\sqrt{-v_p},0\right)^T$ and $r_{0} = \left(0,0,1\right)^T$. Thus we have
\begin{equation}\label{eq:convex p}
   r_{\pm}\nabla \lambda_{\pm}  =\frac{v_{pp}}{2(-v_p)^\frac32}.
\end{equation}
Now, we use $p$, $u$ and $T$ as state variables. Since $v_t - u_{\xi} = v_pp_t + v_TT_t - u_{\xi}$ and $\eta_t = \eta_p p_t + \eta_T T_t$, we can also write system \eqref{eq:Lagrangian-n} under the form
\begin{equation}\label{eq:equations p, T}
  \left\{
  \begin{array}{r}
   p_t  - \frac{\eta_T}{v_p\eta_T - v_T\eta_p}u_{\xi} = 0,\\
   u_t + p_{\xi} = 0,\\
   T_t  + \frac{\eta_p}{v_p\eta_T - v_T\eta_p}u_{\xi} = 0.
  \end{array}
\right.
\end{equation}
At last, we exploit the equation of state \eqref{eq:pressure-ion}. 
\begin{lemma}\label{lem:c}
Under \eqref{eq:pressure-ion} we have
\begin{equation}\label{eq:cc}
\frac{\eta_T}{v_p\eta_T - v_T\eta_p} = - \frac{p^2\left[\frac{5}{2} + \frac{1}{2}\alpha(1 - \alpha)\left(\frac52+\frac{\Ti}{T}\right)^2 \right]}{a^2T(1 + \alpha)\left[\frac{3}{2} + \frac{1}{2}\alpha(1 - \alpha)\left(\frac{15}{4} + \frac{3\Ti }{T} + \frac{\Ti ^2}{T^2}\right)\right]}<0.
\end{equation}
\end{lemma}
\begin{proof}
By \eqref{eq:spec-vol-ion} we compute
\[
 v_p = - \frac{RT}{mp^2}(1+\alpha)\left[1 + \frac{1}{2}\alpha(1 - \alpha) \right],\qquad
  v_T = \frac{R}{mp}(1 + \alpha)\left[1 + \frac{1}{2}\alpha(1 - \alpha)q \right].
\]
Moreover, by \eqref{eq:Maxwell} we have
\begin{equation}
  \eta_p = -\frac{1 + \alpha}{p}\left[1 + \frac{1}{2}\alpha(1 - \alpha)q \right],
         \qquad
  \eta_T = \frac{1 + \alpha}{T}\left[\frac{5}{2} + \frac{1}{2}\alpha(1 - \alpha)q^2 \right].
         \label{eq:eta_pT}
\end{equation}
Then, we obtain
\[
 v_p\eta_T - v_T\eta_p=
  - \frac{a^2(1 + \alpha)^2}{p^2}\left[\frac{3}{2} + \frac{1}{2}\alpha(1 - \alpha)\left(q^2-2q+\frac52\right) \right],
\]
whence \eqref{eq:cc} follows.
\end{proof}
%
We record here, for future reference, that by \eqref{eq:eta_pT} we deduce
\begin{equation}\label{eq:eta_p/eta_T}
 - \frac{\eta_p}{\eta_T} = \frac{T\left[1 + \frac{1}{2}\alpha(1 - \alpha)\qq \right]}{p\left[\frac{5}{2} + \frac{1}{2}\alpha(1 - \alpha)\qq^2 \right]}.
\end{equation}

By Lemma \ref{lem:c} we can introduce the notation
\begin{equation}\label{eq:c}
\lambda^2 
= \lambda^2(p,T) 
= - \frac{\eta_T}{v_p\eta_T - v_T\eta_p}, 
\end{equation}
with $\lambda>0$. Notice that if $\alpha=0$ we recover $\lambda= \gamma p \rho$ with $\gamma=\frac53$, see \cite[(18.5), (18.30)]{Smoller}. Under this notation system \eqref{eq:equations p, T} becomes
\begin{equation}\label{eq:equations-c}
  \left\{
  \begin{array}{rl}
   p_t  + \lambda^2 u_{\xi} = 0,&\\
   u_t + p_{\xi} = 0,&\\
   T_t  - \frac{\eta_p}{\eta_T}\lambda^2 u_{\xi} = 0.&
  \end{array}
\right.
\end{equation}

The proof of the following lemma is straightforward; for brevity it is omitted. 

\begin{lemma}[Eigenvalues and eigenvectors]\label{l:eigen}
The eigenvalues of system \eqref{eq:equations-c} are
\[
\lambda_\pm = \pm \lambda,\qquad \lambda_0=0,
\]
for $\lambda$ defined in \eqref{eq:c}; the corresponding eigenvectors are
\begin{equation}\label{eq:eigenvectors}
R_\pm = \left(
\begin{array}{c}
\pm 1\\
\frac{1}{\lambda}\\[2mm]
\mp\frac{\eta_p}{\eta_T}
\end{array}
\right), \qquad
R_0 = \left(
\begin{array}{c}
0\\
0\\
1
\end{array}
\right).
\end{equation}
The eigenvalue $\lambda_0$ is linearly degenerate; a pair of Riemann invariants for $\lambda_0$  is $\{u,p \}$. A Riemann invariant for both $\lambda_\pm$ is $\eta$.
The eigenvalues of system \eqref{eq:system} are then $u \pm \lambda$ and $u$.
\end{lemma}

%

Now, we investigate the genuine nonlinearity of the eigenvalues $\lambda_\pm$; we refer to \cite{Menikoff-Plohr} for more insight about the failure of this condition. Notice that, when dealing with functions that only depend on $p$ and $T$, we have
$$
     R_{\pm}\nabla = \pm \left(\frac{\partial}{\partial p} - \frac{\eta_p}{\eta_T}\frac{\partial}{\partial T}\right).
$$

\begin{lemma}\label{lem:charcteristic alpha}
We have
\begin{equation}\label{eq:charcteristic alpha}
  R_{\pm}\nabla \alpha 
  = \frac{\pm \frac{1}{2}\alpha(1 - \alpha^2) \frac{\Ti }{T}}{p\left[\frac{5}{2} + \frac{1}{2}\alpha(1 - \alpha)\left(\frac{5}{2} + \frac{\Ti }{T}\right)^2 \right]}.
\end{equation}
\end{lemma}
\begin{proof}
We simply compute
\begin{eqnarray*}
  \frac{\partial \alpha}{\partial p} - \frac{\eta_p}{\eta_T}\frac{\partial \alpha}{\partial T}
  &=& -\frac{\alpha(1 - \alpha^2)}{2p}
  + \frac{T\left[1 + \frac{1}{2}\alpha(1 - \alpha)q\right]}{p\left[\frac{5}{2}
      + \frac{1}{2}\alpha(1 - \alpha)q^2 \right]}
  \times \frac{\alpha(1 - \alpha^2)}{2T}q\\
  &=&  -\frac{\alpha(1 - \alpha^2)}{2p}\left[\frac{\frac{5}{2}
      + \frac{1}{2}\alpha(1 - \alpha)q^2
- q + \frac{1}{2}\alpha(1 - \alpha)q^2}{\frac{5}{2}
      + \frac{1}{2}\alpha(1 - \alpha)q^2}\right],
\end{eqnarray*}
whence \eqref{eq:charcteristic alpha}.
\end{proof}

We notice that Lemma \ref{lem:charcteristic alpha} shows that $\alpha$ is not a Riemann invariant for the eigenvalues $\lambda_\pm$.

\begin{proposition}[Genuine nonlinearity]\label{prop:genuine_nonlinearity}
The eigenvalues $\lambda_\pm$ are genuinely nonlinear if $\alpha \ge \frac12$. If $\alpha<\frac12$ they are no more genuinely nonlinear for suitably small values of $p$ and $T$.
\end{proposition}
\begin{proof}
We focus on $\lambda_+$ since $R_+\nabla\lambda_+ = R_-\nabla\lambda_-$. Since $R_{+} \nabla \log \lambda_+ = \frac{R_{+} \nabla \lambda_+}{\lambda_+}$, the eigenvalue $\lambda_+$ is genuinely nonlinear if and only if $R_{+} \nabla \log \lambda_+>0$. We compute
\begin{align}
  \log \lambda_+
& = \log p - \frac{1}{2} \log T - \frac{1}{2} \log(1 + \alpha) - \log a  +  \frac{1}{2}\log A - \frac{1}{2}\log B,\label{eq:lambda+2}
\end{align}
for
\begin{equation}\label{eq:AB}
A  =  \frac{5}{2} + \phi q^2,\qquad
B = \frac{3}{2} + \phi\left(q^2-2q+\frac52\right) = \frac32 + \phi\left(\tau^2 + 3\tau + \frac{15}{4}\right),\qquad
\phi = \frac12\alpha(1-\alpha).
\end{equation}
In the following, we exploit many times identity \eqref{eq:eta_p/eta_T}; we split the proof into three steps.

\paragraph{(1)} About the first three terms in the right-hand side of \eqref{eq:lambda+2} we claim that
\begin{equation}
R_+\nabla\left[\log p - \frac{1}{2}\log T - \frac{1}{2}\log (1 + \alpha)\right]
  = \frac{1}{p A}
\left[2+ \phi\left(5 + \frac{4\Ti }{T} + \frac{\Ti ^2}{T^2}\right)\right].\label{eq:logp - log T - log (1 + alpha)}
\end{equation}

Indeed, we have
\[
\left(\frac{\partial }{\partial p} - \frac{\eta_p}{\eta_T}\frac{\partial }{\partial T}\right)\left(\log p - \frac{1}{2}\log T\right) =\ \frac{1}{p A}\left[2+ \phi q\left(2 + \frac{\Ti }{T}\right)\right],
\]
and, by Lemma \ref{lem:charcteristic alpha}, we have
\[
- \frac{1}{2}\left(\frac{\partial }{\partial p} - \frac{\eta_p}{\eta_T}\frac{\partial }{\partial T}\right)\log (1 + \alpha)
\ = \  - \frac{1}{2(1 + \alpha)}
\, \frac{\frac{1}{2}\alpha(1 - \alpha^2) \frac{\Ti }{T}}{pA}\ = \ -
\frac{\phi }{2pA}\frac{\Ti }{T}.
\]
Then \eqref{eq:logp - log T - log (1 + alpha)} follows.

\paragraph{(2)} About the last two terms in \eqref{eq:lambda+2}, we notice that
\begin{align*}
R_+\nabla\left(\frac{1}{2}\log A - \frac{1}{2}\log B\right)
&=  \frac12\left(\frac12-\alpha\right)\left[\frac{1}{A}q^2-\frac{1}{B}\left(q^2-2q+\frac52\right)
\right]R_+\nabla\alpha +\phi\frac{\eta_p}{\eta_T}\left[
\frac{q-1}{B} - \frac{q}{A}\right]\partial_T q
\\
&=  I + I\!I.
\end{align*}
About $I$ we notice that
\[
\frac{1}{2A}q^2-\frac{1}{2B}\left(q^2-2q+\frac52\right) = -\frac{(q-\frac52)^2}{AB} = -\frac{\frac{\Ti^2}{T^2}}{AB},
\]
so that, by Lemma \ref{lem:charcteristic alpha} and \eqref{eq:AB} we have
\[
I = -\frac{\alpha(1-\alpha^2)\left(\frac12-\alpha\right)\frac{\Ti^3}{T^3}}{4 p A^2 B}.
\]
About $I\!I$ we have that
\[
I\!I = -\frac12\alpha(1-\alpha)\frac{\eta_p}{\eta_T}\frac{1}{AB}\frac{\Ti}{T^2}\left[
qA-A-qB\right].
\]
Since $qA-A-qB 
= \frac{\Ti}{T}\left[1+\phi q\right]
$
then, by \eqref{eq:eta_p/eta_T}, we deduce
\[
I\!I  =  -\phi\frac{\eta_p}{\eta_T}\frac{1}{A B}\frac{\Ti^2}{T^3}\left[1+\phi q\right]
 =  \phi \frac{T\left[1 + \phi q \right]}{p\left[\frac{5}{2} + \phi q^2 \right]} \frac{1}{A B}\frac{\Ti^2}{T^3}\left[1+\phi q\right]
 =
\phi \frac{1}{p A^2 B}\frac{\Ti^2}{T^2}\left[1+\phi q\right]^2.
\]
As a consequence
\begin{equation}\label{e:RR22}
R_+\nabla\left(\frac{1}{2}\log A - \frac{1}{2}\log B\right) =
\phi \frac{1}{pA^2B}\frac{\Ti^2}{T^2}\left[
\left(1+\phi q\right)^2 - \frac12(1+\alpha)\left(\frac12-\alpha\right)\frac{\Ti}{T}
\right].
\end{equation}

\paragraph{(3)} By \eqref{eq:logp - log T - log (1 + alpha)} and \eqref{e:RR22} we deduce that
\begin{equation}
R_+\nabla\log\lambda_+
 = 
\frac{2}{p A}\left\{
1 + \frac{\phi}{2}\left(q^2-q+\frac54\right)
+
\frac{\phi}{2A B}\tau^2\left[
\left(1+\phi q\right)^2 - \frac12(1+\alpha)\left(\frac12-\alpha\right)\tau
\right]\right\}.
\label{eq:gnl-last}
\end{equation}
The right-hand side of \eqref{eq:gnl-last} is positive for $\alpha\ge\frac12$. We claim that \eqref{eq:gnl-last} is negative for suitable small values of $T$ and $p=p(T)$. To prove the claim, we take $T$ sufficiently small and fix $p=p(T)$ in such a way that
\begin{equation}\label{eq:alphatau2}
\alpha\tau^2 = 1.
\end{equation}
Therefore, by using notation as in \eqref{eq:deg-ionization},
\begin{equation}\label{eq:pressure-tau}
p = p(T)\sim\frac{\tau^4 -1}{\kappa}T^{\frac52}e^{-\tau}
\end{equation}
for $T\to0$ and so $p$ is also small. For $T\to0$ we have $\alpha\sim0$, $\alpha q\sim0$, $A\sim3$, $B\sim2$ and then by \eqref{eq:gnl-last} and \eqref{eq:alphatau2} for $\alpha<\frac12$ we have
\begin{equation}\label{eq:gnl-tau}
R_+\nabla\log\lambda_+
 \sim
\frac{1}{3p}\left\{
2+ \frac12
+
\frac{1}{12}\left[
1 - \frac\tau4
\right]\right\} = \frac{1}{144p}\left(124-\tau\right)<0
\end{equation}
if $\tau>124$.
\end{proof}
\begin{nb}\label{rem:genuine_nonlinearity}
The value $T_{\rm i}/T>124$ found in \eqref{eq:gnl-tau} corresponds, for the value $T_{\rm i}= 1.578\times 10^5${\rm K}, see Remark \ref{rem:values}, to $T\lesssim1.2\times 10^3${\rm K}. Notice that the term in braces in \eqref{eq:gnl-last} tends exponentially to $1$ as $T\to0$ for every fixed $p>0$. Moreover, it also tends to $1$ as $p\to0$ for every fixed $T>0$.

%

\end{nb}


\subsection{The inflection locus}\label{subsec:inflection loci}
Now, we study in greater detail the {\em inflection locus}, which is the set 
\[
\mathcal{I}= \left\{(\alpha, T);\ R_\pm\nabla \lambda_{\pm} = 0, \, 0 < \alpha < 1,\ T > 0\right\}.
\]
Since $R_+\nabla \lambda_{+} = R_-\nabla \lambda_{-}$, either cases lead to the same result. 

\begin{lemma}\label{lem:expression f(alpha,T)} The inflection locus $\mathcal{I}$ is the zero set of the function
\begin{eqnarray}
f(\alpha,T)
&=& 
 1 + \alpha(1 - \alpha)\left(\frac{5}{4} + \frac{\Ti}{T} + \frac{\Ti^2}{4T^2}\right) \nonumber\\
& &{} + \frac{\alpha(1 - \alpha)\frac{\Ti^2}{15T^2}\left[1 + \frac{1}{2}\alpha(1 - \alpha)\left(\frac{5}{2} + \frac{\Ti}{T}\right)\right]^2 - \alpha(1 - \alpha^2)\left(1  - 2\alpha\right)\frac{\Ti^3}{60T^3}}{\left[1 + \alpha(1 - \alpha)\left(\frac{5}{4} + \frac{\Ti}{T} + \frac{\Ti^2}{5T^2}\right) \right]\left[1 + \alpha(1 - \alpha)\left(\frac{5}{4} + \frac{\Ti}{T} + \frac{\Ti^2}{3T^2}\right)\right]}.
\label{eq:f}
\end{eqnarray}
\end{lemma}
%
\begin{proof}
Formula \eqref{eq:f} follows by \eqref{eq:gnl-last} with a direct calculation: the function $f$ is the term in braces in that formula.
\end{proof}
\par

Formula  \eqref{eq:gnl-last} shows that $R_+\nabla \log \lambda_{+} > 0$ if and only if  $f(\alpha, T) > 0.$  
We refer to Figure \ref{fig:inflection_locus} for the locus $\mathcal{I}$. Notice that significant values of the temperature range from $0$ to $1.5\cdot10^3$K; on the other hand, the ionization degree is low, ranging from $0$ to $2\cdot10^{-4}$. Both ranges are in a physical range.

\begin{figure}[htbp]	
  \centering
  \begin{tabular}{c}
  \includegraphics[width =.65\linewidth]{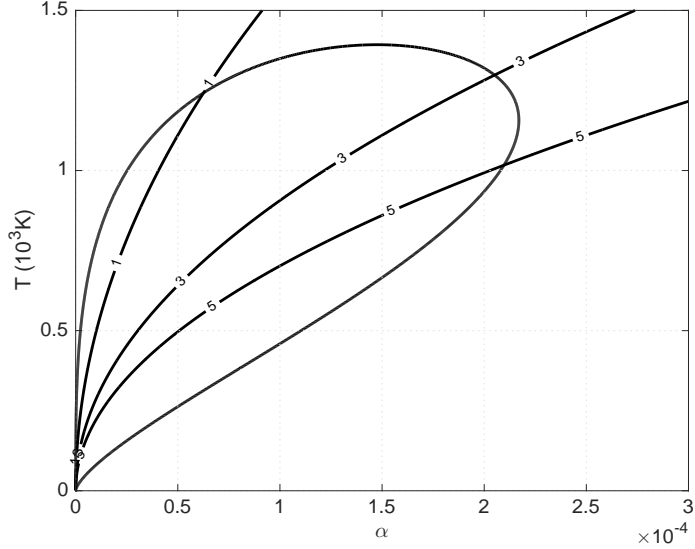}
  \end{tabular}
\caption{The inflection locus with the level curves $\alpha(T_{\rm i}/T)^2=C$, $C=1,3,5$. The level curve with $C=1$ was used in the proof of Proposition \ref{prop:genuine_nonlinearity}.}
  \label{fig:inflection_locus}
\end{figure}


\begin{proposition}\label{prop:inflection locus branch}
The inflection locus is an algebraic curve having a singularity at $(\alpha,T)= (0,0)$. Near this point there exist two branches, whose behavior for $T\to0$ is, respectively, 
$$
    \alpha \sim  60 \left(\frac{T}{\Ti}\right)^{\!3}\quad \hbox{ and }\quad  \alpha \sim  \left(\frac{T}{\Ti}\right)^{\!\frac{3}{2}}.
$$
\end{proposition}
\begin{proof}
 We find by the proof of Proposition \ref{prop:genuine_nonlinearity} that $\nabla_+\log \lambda_+$ changes its sign along the curve $\alpha \tau^2 = c \ (c > 0).$ Since, by Lemma \ref{lem:expression f(alpha,T)}, the inflection locus is an algebraic curve in the $(\alpha,T)$-plane, the inflection locus behaves like $\alpha\tau^2 \to 0$ or $\alpha\tau^2 \to \infty$ as $(\alpha, T) \to (0,0).$

First, suppose that $\alpha\tau^2 \to 0$ as $(\alpha, T) \to (0,0)$; then also $\alpha\tau \to 0$. About the function $f$ we have that
\begin{gather*}
\phi\left(q^2-q+\frac54\right)=o(1),\quad A\sim\frac52,\quad B\sim\frac32,\quad
\frac{\phi}{A B}\tau^2 \sim \frac{2}{15}\alpha\tau^2,
\\
\left(1+\phi q\right)^2=o(1),\quad \frac12(1+\alpha)\left(\frac12-\alpha\right)\tau \sim\frac14\tau.
\end{gather*}
Only the first and the last summand in the expression of $f$ give nonzero contributions and equation $f(\alpha,T)=0$ can be written as
$
     2 - \frac{\alpha}{30}\tau^3 \sim 0.
$
Thus we have $\alpha \sim 60 \tau^{-3}$; this is the left branch on Figure \ref{fig:inflection_locus}.
Next, suppose that $\alpha\tau^2 \to \infty$  as $(\alpha, T) \to (0,0)$. Then
\begin{gather*}
\phi\left(q^2-q+\frac54\right)\sim\frac12\alpha\tau^2,\quad
A\sim\frac12\alpha\tau^2,\quad B\sim\frac12\alpha\tau^2,\quad
\frac{\phi}{A B}\tau^2 \sim \frac{2}{\alpha\tau^2},
\\
\left(1+\phi q\right)^2\sim 1+ \alpha\tau + \frac{\alpha}{4}\,\alpha\tau^2,
\quad \frac12(1+\alpha)\left(\frac12-\alpha\right)\tau \sim\frac14\tau.
\end{gather*}
In this case we have that equation $f(\alpha,T)=0$ becomes
\begin{equation}\label{eq:sim0}
\frac12\alpha\tau^2 + \frac{2}{\alpha\tau^2}\left(1+\frac12\alpha\tau\right)^2 - \frac{1}{2\alpha\tau} \sim0.
\end{equation}
Expression \eqref{eq:sim0} cannot hold if $\alpha\tau \sim C\ne0$; then, either $\alpha\tau \sim 0$ or $\alpha\tau\sim\infty.$ In  case $\alpha\tau \sim \infty$ expression \eqref{eq:sim0} can be written as
$
  \frac12\alpha\tau^2 - \frac{1}{2\alpha\tau} + \frac{\alpha}{2} \sim 0,
$
which is impossible. In the case $\alpha\tau \sim 0$ equation \eqref{eq:sim0} becomes
$$
  \frac12\alpha\tau^2 - \frac{1}{2\alpha\tau}\sim 0
$$
and consequently
$
 \alpha \sim \tau^{-\frac32}.
$
This is the right branch on Figure \ref{fig:inflection_locus}, left. 
\end{proof}
\par
From the proof of Proposition \ref{prop:inflection locus branch} we deduce that $\nabla R_\pm\lambda_\pm$ is positive outside the bounded region of Figure \ref{fig:inflection_locus} and negative inside it. 

We now provide an explicit sufficient condition on the genuine nonlinearity of the eigenvalues $\lambda_\pm$.

\begin{theorem}\label{thm:genuine nonlinearity}
If either $\alpha \leq 60 \left(\frac{T}{\Ti}\right)^{\! 3}$ or $\frac{\Ti}{T} \leq 54.5375,$
then each eigenvalue $\lambda_\pm$ is genuinely nonlinear.
\end{theorem}
\begin{proof}
With the aid of the inequality $(1 + a)(1 + b)(1 + c) \geq 1 + a + b + c$, for $a,b,c \geq 0, $ 
we have
\begin{eqnarray*}
f(\alpha,T)& > & \frac{ 1 + \alpha(1 - \alpha)\left[\frac{15}{4} + \frac{3\Ti}{T} + \left(\frac{1}{4} + \frac{1}{5}  + \frac{1}{3}\right)\frac{\Ti^2}{T^2}\right]- \frac{\alpha(1  - 2\alpha)\Ti^3}{60T^3} + \frac{\alpha(1 - \alpha)\Ti^2}{15T^2}}{\left[1 + \alpha(1 - \alpha)\left(\frac{5}{4} + \frac{\Ti}{T} + \frac{\Ti^2}{5T^2}\right) \right]\left[1  + \alpha(1 - \alpha)\left(\frac{5}{4} + \frac{\Ti}{T} + \frac{\Ti^2}{3T^2}\right)\right]} \\
& = & \frac{1 + \alpha(1 - \alpha)\left[\frac{15}{4} + \frac{3\Ti}{T} + \frac{51}{60}\left(\frac{\Ti}{T}\right)^{\! 2}\right] - \frac{\alpha(1  - 2\alpha)}{60}\left(\frac{\Ti}{T}\right)^{\! 3}}
  {\left[1 + \alpha(1 - \alpha)\left(\frac{5}{4} + \frac{\Ti}{T} + \frac{\Ti^2}{5T^2}\right) \right]\left[1  + \alpha(1 - \alpha)\left(\frac{5}{4} + \frac{\Ti}{T} + \frac{\Ti^2}{3T^2}\right)\right]}.
\end{eqnarray*}
Then, both characteristic eigenvalue $\lambda_\pm$ are genuinely nonlinear if the following inequality holds:
\begin{equation}\label{eq: condition GN}
  \alpha(1  - 2\alpha)\left(\frac{\Ti}{T}\right)^{\! 3} \leq  60 + \alpha(1  - \alpha)\left[225 + 180\left(\frac{\Ti}{T}\right) + 51\left(\frac{\Ti}{T}\right)^{\! 2}\right].
\end{equation}

Now, assume that the first condition in the statement of the theorem holds; then 
$$
    \alpha(1  - 2\alpha)\left(\frac{\Ti}{T}\right)^{\! 3} < \alpha\left(\frac{\Ti}{T}\right)^{\! 3} \leq 60
$$
showing that inequality \eqref{eq: condition GN} is satisfied. Next, note that
$$
    \frac{60}{\alpha(1 - 2\alpha)} \geq 480, \quad
    \frac{1 - \alpha}{1 - 2\alpha} \geq 1 \quad
    \left({\textstyle 0 \leq \alpha < \frac{1}{2}}\right).
$$
Then \eqref{eq: condition GN} is true if 
$$
    \left(\frac{\Ti}{T}\right)^{\! 3} \leq 705 + 180\left(\frac{\Ti}{T}\right) + 51\left(\frac{\Ti}{T}\right)^{\! 2}.
$$
Since $x^3 - 705 - 180 x - 51 x^2 < 0$ for $0 \leq x \leq 54.5375,$ the claim follows.
\end{proof}
\begin{nb}\label{nb:genuine nonlinearlity_2}
For monatomic hydrogen gas, we have $\Ti = 1.578\times 10^5\, {\rm K}.$ Thus the $\lambda_\pm$ characteristic fields are genuinely nonlinear if $T \geq 2.8934\times 10^3 {\rm K}.$
\end{nb}

\section{Hugoniot Loci}\label{sec:Hugoniot}

For the gas dynamic equations \eqref{eq:system}, the Rankine-Hugoniot conditions for a discontinuity of constant speed $s$ are
\begin{equation}\label{eq:RH}
\left\{
\begin{array}{l}
s[\rho] = [\rho u],
\\
s\ds[\rho u] = [\rho u^2 + p],
\\
s\ds[\rho E] = [\rho u E + pu].
\end{array}
\right.
\end{equation}
Here, as usual, we denoted $[\rho]=\rho_+-\rho_-$, where $\rho_\pm$ are the traces of $\rho$ at $x=st$ from the right and from the left, respectively; the same notation is used for the other variables. If $[\rho]=0$ then $[u]=0$ by $\eqref{eq:RH}_1$ and $[p]=0$ by $\eqref{eq:RH}_2$; in this case, conditions \eqref{eq:RH} describe the jumps of the state variables at a contact discontinuity and $s=u_\pm$. From now on we focus on the much more interesting case of shock waves corresponding to eigenvalues $\lambda_\pm$ and assume $[\rho]\ne0$. In this case we can eliminate $s$ from the first equation and substituting it in the other two equations, the conditions \eqref{eq:RH} can be written as
\begin{equation}\label{eq:Rankine-Hugoniot}
\left\{\begin{array}{rcl}
        (u_+ - u_-)^2  + (p_+ - p_-)(v_+ - v_-) & = & 0,\\[1mm]
        e_+ -  e_- + \frac{1}{2}(p_+ + p_-)(v_+ - v_-) & = & 0.
                \end{array}\right.
\end{equation}
In the following we slightly change notation: with reference to \eqref{eq:Rankine-Hugoniot}, we drop the $+$ index and write $0$ instead of the index $-$. Under this notation, we have the equation of {\it Hugoniot locus} of $(\alpha_0,u_0,T_0)$
\begin{equation}\label{eq:Hugoniot-locus}
\left\{\begin{array}{rcll}
        (u - u_0)^2  + (p - p_0)(v - v_0) & = & 0, &\text{kinetic part,}\\[1mm]
        e -  e_0 + \frac{1}{2}(p + p_0)(v - v_0) & = & 0, &\text{thermodynamic part.}
                \end{array}\right.
\end{equation}
For brevity we call $\eqref{eq:Hugoniot-locus}_1$ the {\em kinetic part} and $\eqref{eq:Hugoniot-locus}_2$ the {\em thermodynamic part} of the Hugoniot locus. The aim of this section is to give a precise description of this locus and to evaluate, in particular, the variation of the thermodynamical variables across it; this analysis is fundamental for the study of shock waves, see \cite{Asakura-Corli_reflected}.



\subsection{The variation of $T$}
In this section we begin the study of the Hugoniot locus by focusing first on the variation of the temperature. Then we study the thermodynamic part of this locus in the $(\alpha,T)$-plane. For a contact discontinuity, the thermodynamical part of the Hugoniot locus is given by $p = p_0$ and $\alpha = \left(1 + \kappa p_0T^{-\frac{5}{2}}e^{\frac{\Ti }{T}}\right)^{-\frac{1}{2}}$, see \eqref{eq:deg-ionization} , which is excluded in the following discussion.

\begin{proposition}[Variation of the temperature]\label{prop:Hugoniot curve}
The thermodynamic part of the Hugoniot locus \eqref{eq:Hugoniot-locus} of $(\alpha_0,T_0)$ can be expressed as
\begin{equation}\label{eq:T_+/T_-}
    \frac{T}{T_0} = \frac{\left(4 + \frac{p}{p_0}\right)(1 + \alpha_0) + 2\frac{\Ti}{T_0}\alpha_0}{\left(4 + \frac{p_0}{p}\right)(1 + \alpha) + 2\frac{\Ti}{T}\alpha}
\end{equation}
or else, by using only $\alpha$ and $T$, as
\begin{eqnarray}
& & T\left[4(1 + \alpha) + \left(\frac{1 - \alpha_0^2}{1 - \alpha}\right) \left(\frac{\alpha}{\alpha_0}\right)^2\left(\frac{T_0}{T}\right)^{\frac{5}{2}}
e^{-\frac{\Ti}{T_0} + \frac{\Ti}{T}} + 2\frac{\Ti}{T}\alpha\right]
\nonumber
\\
   &=& T_0\left[4(1 + \alpha_0) + \left(\frac{1 - \alpha^2}{1 - \alpha_0}\right)\left(\frac{\alpha_0}{\alpha}\right)^2 \left(\frac{T}{T_0}\right)^{\frac{5}{2}}
e^{-\frac{\Ti}{T} + \frac{\Ti}{T_0}} +  2\frac{\Ti}{T_0}\alpha_0\right].\label{eq:T_+/T_-prop}
\end{eqnarray}
\end{proposition}

\begin{proof}
By using the enthalpy, equation $\eqref{eq:Hugoniot-locus}_2$ can be written as
\begin{equation}\label{eq:RH thermodynamic}
        H -  H_0 - \frac{1}{2}(p - p_0)(v + v_0) = 0.
\end{equation}
Now, by \eqref{eq:pressure-ion} and \eqref{eq:e}, we have
$ p v  =  a^2T(1 + \alpha)$ and $H = 
    \frac{5}{2}pv + a^2\Ti\alpha$.
As a consequence, we deduce 
\[
H -  H_0 - \frac{1}{2}(p - p_0)(v + v_0) =  2(pv - p_0v_0) + \frac{1}{2}(vp_0 - v_0p) + a^2\Ti(\alpha - \alpha_0).
\]
Then, condition \eqref{eq:RH thermodynamic} is equivalent to
\[
2\left[a^2T(1 + \alpha) - a^2T_0(1 + \alpha_0)\right] + \frac{1}{2}\left[a^2T\frac{p_0}{p}(1 + \alpha) - a^2T_0\frac{p}{p_0}(1 + \alpha_0)\right]  + a^2\Ti(\alpha - \alpha_0)
= 0.
\]
If we divide the equation above by $a^2T_0$ 
we obtain \eqref{eq:T_+/T_-}.
At last, by \eqref{eq:p-alphaT} we have
\begin{equation}\label{eq:p_+/p_-}
  \frac{p}{p_0} = \left(\frac{1 - \alpha^2}{1 - \alpha_0^2}\right)\left(\frac{\alpha_0}{\alpha}\right)^2
  \left(\frac{T}{T_0}\right)^{\frac{5}{2}}e^{-\frac{\Ti}{T} + \frac{\Ti}{T_0}}.
\end{equation}
We deduce \eqref{eq:T_+/T_-prop} by inserting \eqref{eq:p_+/p_-} into \eqref{eq:T_+/T_-}.
\end{proof}

\begin{nb}[Polytropic gas]\label{nb:polytropic}
We explicitly compute $T/T_0$ in the case of a polytropic gas, see \eqref{eq:T_+/T_-}. By \eqref{eq:pvT} we obtain $T/T_0 = \left(p/p_0\right)^{\!\frac{\gamma-1}{\gamma}} e^{\frac{\gamma-1}{\gamma}(\eta - \eta_0)}.$ In the monatomic case $\gamma = \frac{5}{3},$ we have
$$
  \frac{T}{T_0}
= \left(\frac{p}{p_0}\right)^{\!\frac{2}{5}}e^{\frac{2}{5}(\eta - \eta_0)}
= \frac{4 + \frac{p}{p_0}}{4 + \frac{p_0}{p}},
$$
which coincides with \eqref{eq:T_+/T_-} for $\alpha=\alpha_0=0$. 
\end{nb}

\par
The following proposition which will be crucial in the following.
\begin{proposition}[$T$ as a function of $\alpha$]\label{eq:Hugoniot curve alpha T}
In the $(\alpha, T)$-plane, the thermodynamic part \eqref{eq:T_+/T_-prop} of the Hugoniot locus \eqref{eq:Hugoniot-locus} of $(\alpha_0,T_0)$ is the graph of a strictly increasing function $T=T(\alpha)$, for $\alpha\in (0,1).$ Moreover,
\begin{equation}\label{eq:Tasymptotics}
\lim_{\alpha\to0}T(\alpha)=0, \qquad \lim_{\alpha\to1}T(\alpha)=\infty.
\end{equation}
\end{proposition}

\begin{figure}[htbp]	
  \centering
  \begin{tabular}{c}
\includegraphics[width =.65\linewidth]{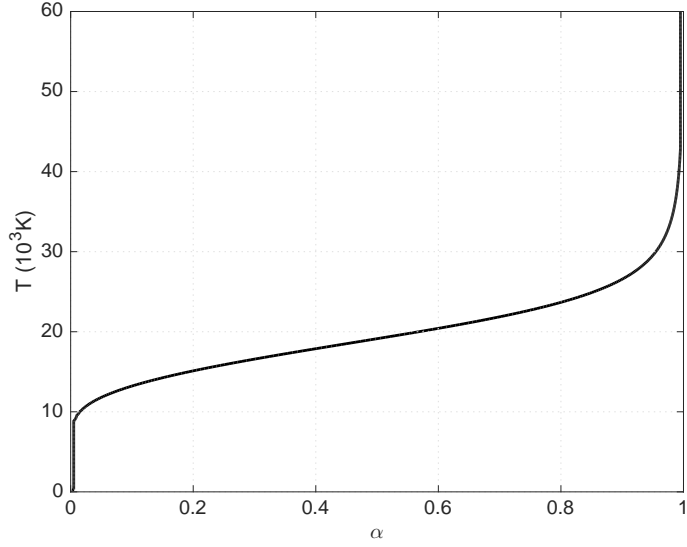}
  \end{tabular}
\caption{The thermodynamic part of the Hugoniot locus and the function $T=T(\alpha)$, see Proposition \ref{eq:Hugoniot curve alpha T}. Here 
$T_0=300{\rm K}$, $\alpha_0 = 3.5929\times 10^{-114}$, $p_0=1466.3{\rm Pa}$; see the Appendix for these values, which are also used in the following figures.}
  \label{fig:HugoniotT}
\end{figure}

\begin{proof}
We refer to Figure \ref{fig:HugoniotT} for a graph of the function $T$. We split the proof into some steps.

\smallskip

\begin{enumerate}[{\em (i)}]

\item By \eqref{eq:T_+/T_-prop}, the thermodynamic part of the Hugoniot locus is the set defined by $F(\alpha,T)=0$, for 
\begin{eqnarray}
  F(\alpha,T)&= & T\left[4(1 + \alpha) + \left(\frac{1 - \alpha_0^2}{1 - \alpha}\right) \left(\frac{\alpha}{\alpha_0}\right)^2\left(\frac{T_0}{T}\right)^{\frac{5}{2}}
e^{-\frac{\Ti}{T_0} + \frac{\Ti}{T}} + 2\frac{\Ti}{T}\alpha\right]
\nonumber
\\
   & & - T_0\left[4(1 + \alpha_0) + \left(\frac{1 - \alpha^2}{1 - \alpha_0}\right)\left(\frac{\alpha_0}{\alpha}\right)^2 \left(\frac{T}{T_0}\right)^{\frac{5}{2}}
e^{-\frac{\Ti}{T} + \frac{\Ti}{T_0}} +  2\frac{\Ti}{T_0}\alpha_0\right].\label{eq:F}
\end{eqnarray}
By differentiating \eqref{eq:F} with respect to $T,$ and then introducing $p$ and $p_0$, we compute
\begin{eqnarray}
F_T(\alpha,T)&=& 4(1 + \alpha )
- \left(\frac{1 - \alpha_0^2}{1 - \alpha }\right)\left(\frac{\alpha }{\alpha_0}\right)^{\!2}\left(\frac{T_0}{T}\right)^{\frac{5}{2}}\left(\frac{3}{2} + \frac{\Ti}{T}\right)e^{-\frac{\Ti}{T_0} + \frac{\Ti}{T}}\nonumber\\
&&
-  \left(\frac{1 - \alpha ^2}{1 - \alpha_0}\right)\left(\frac{\alpha_0}{\alpha}\right)^{\!2}
\left(\frac{T}{T_0}\right)^{\frac{3}{2}}\left(\frac{5}{2} + \frac{\Ti}{T}\right)e^{-\frac{\Ti}{T} + \frac{\Ti}{T_0}}\nonumber\\
&=& -\left[(1 + \alpha)\frac{p_0}{p}\left(\frac{3}{2} + \frac{\Ti}{T}\right)  + (1 + \alpha_0)\frac{p}{p_0}\frac{T_0}{T}\left(\frac{5}{2} + \frac{\Ti}{T}\right) - 4(1 + \alpha )\right]\label{eq:F_T}
\\
& = & -4(1 + \alpha)\Phi(\alpha,T),
\nonumber
\end{eqnarray}
where we defined
\begin{equation}\label{e:Phi-first}
 \Phi(\alpha, T) = \frac{1}{4}\frac{p_0}{p}\left(\frac{3}{2} + \frac{\Ti}{T}\right) + \frac{1}{4}\frac{p}{p_0}\frac{T_0(1 + \alpha_0)}{T(1 + \alpha)}\left(\frac{5}{2} + \frac{\Ti}{T}\right) - 1.
\end{equation}
Obviously, $F_T(\alpha,T) \neq 0$ if and only if $\Phi(\alpha,T) \neq 0.$

\item We consider
$0 < \alpha < 1, T > 0$ and $F(\alpha, T) = 0.$ We claim that 
\begin{equation}\label{eq:Phipositive}
\Phi(\alpha, T) > 0\quad \hbox{ if \quad either $0 < \alpha < \alpha_0, \,T \leq T_0$ \quad or $\alpha \geq \alpha_0.$}
\end{equation}
Indeed, by Proposition \ref{prop:Hugoniot curve} the condition $F(\alpha,T)=0$ is equivalent to \eqref{eq:T_+/T_-}, namely,
\begin{equation}\label{eq:THH}
\frac{1}{4}\frac{p}{p_{0}}\frac{T_{0}(1 + \alpha_{0})}{T(1 + \alpha)}
=  \frac{1}{4}\frac{p_{0}}{p} +  \frac{\Ti(\alpha - \alpha_{0})}{2T(1 + \alpha)} - \frac{T_0(1 + \alpha_0)}{T(1 + \alpha)} + 1.
\end{equation}
Then, if we denote
\begin{equation}\label{e:DeltaD}
   \Delta = 1 - \frac{T_0(1 + \alpha_0)}{T(1 + \alpha)}  + \frac{\Ti(\alpha - \alpha_{0})}{2T(1 + \alpha)},
\end{equation}
we have, for $\alpha, T$ satisfying $F(\alpha,T) = 0,$
$$
 \Phi(\alpha,T) = \frac{p_0}{p}\left(1 + \frac{\Ti}{2T}\right) + \left(\frac{5}{2} + \frac{\Ti}{T}\right) \Delta - 1.
$$
By \eqref{eq:THH} and \eqref{e:DeltaD}, the quotient  $\Pi = \frac{p_{0}}{p}$ satisfies the equation
\begin{equation}\label{eq:quadratic}
\Pi^{2} + 4\Delta \Pi -  \frac{T_{0}(1 + \alpha_{0})}{T(1 + \alpha)}= 0
\end{equation}
and then
\begin{equation}
	\Pi = \sqrt{4\Delta^{2} + \frac{T_{0}(1 + \alpha_{0})}{T(1 + \alpha)}} - 2\Delta.
\end{equation}
As a consequence, we deduce
$$
   \Phi(\alpha,T) = \left(1 + \frac{\Ti}{2T}\right)\sqrt{4\Delta^{2} + \frac{T_{0}(1 + \alpha_{0})}{T(1 + \alpha)}} + \frac{\Delta}{2} - 1.
$$

If $\Delta > 2, $ then $\Phi(\alpha,T) > 0$ for every $\alpha$ and $T$; hence, assume $\Delta \leq 2$. If $\Phi(\alpha,T) = 0$, then $\Delta$ satisfies the quadratic equation
\begin{equation} \label{eq:Delta}
  \left(\frac{15}{4} + \frac{4\Ti}{T} + \frac{\Ti^2}{T^2}\right)\Delta^{2}   +  \Delta  + \left(1 + \frac{\Ti}{2T}\right)^{\! 2}\frac{T_{0}(1 + \alpha_{0})}{T(1 + \alpha)} -  1  = 0.
\end{equation}
By \eqref{e:DeltaD}, equation \eqref{eq:Delta} can be written as
$$
   \left(\frac{15}{4} + \frac{4\Ti}{T} + \frac{\Ti^2}{T^2}\right)\Delta^{2}  +  \frac{\Ti}{T}\left[\left(1 + \frac{\Ti}{4T}\right)\frac{T_{0}(1 + \alpha_{0})}{T(1 + \alpha)} +  \frac{\alpha - \alpha_{0}}{2(1 + \alpha)}\right]  = 0.
$$
If $\alpha \geq\alpha_0,$ then the left-hand side of the above expression is strictly positive, which is a contradiction. If $0 < \alpha < \alpha_0,$ we have
$$
    \left(1 + \frac{\Ti}{4T}\right)\frac{T_{0}(1 + \alpha_{0})}{T(1 + \alpha)} +  \frac{\alpha - \alpha_{0}}{2(1 + \alpha)} 
> \frac{1}{1 + \alpha}\left[\frac{T_{0}}{T}(1 + \alpha_{0}) +  \frac{\alpha - \alpha_{0}}{2}\right].
$$
If $T \leq T_0,$ the above quantity is positive and hence we reach a contradiction again. This proves \eqref{eq:Phipositive}.
%

\item Consider the function $F$ defined in \eqref{eq:F}; clearly $F(\alpha_0,T_0) = 0$ and $F_T(\alpha_0,T_0) \neq 0$ by \eqref{eq:Phipositive}. Then it follows from the Implicit Function Theorem that $T$ is a function of $\alpha$ in a neighbourhood of $\alpha = \alpha_0;$ moreover, $dT/d\alpha=-F_\alpha(\alpha,T)/F_T(\alpha,T)$.
By differentiating \eqref{eq:F} with respect to $\alpha$ and introducing $p$ and $p_0$, we have
\begin{align}
F_{\alpha}(\alpha,T) &=  \frac{2T_0\alpha_0^2}{(1 - \alpha_0)}\left(\frac{1}{\alpha^3}\right)\left(\frac{T }{T_0}\right)^{\frac{5}{2}}
e^{-\frac{\Ti}{T } + \frac{\Ti}{T_0}} - T\left[4 + \frac{(1 - \alpha_0^2)}{\alpha_0^2}\left[\frac{\alpha(2 - \alpha)}{(1 - \alpha)^{2}}\right]\left(\frac{T_0}{T }\right)^{\frac{5}{2}}
e^{-\frac{\Ti}{T_0} + \frac{\Ti}{T }} +  2\frac{\Ti}{T }\right] \nonumber\\
&= \frac{2p}{p_0}\frac{T_0(1 + \alpha_0)}{\alpha(1-\alpha^2)}
  + T\left[4 + \frac{(1+\alpha)(2-\alpha)}{\alpha(1-\alpha)}\frac{p_0}{p} +  2\frac{\Ti}{T }\right].\label{eq:Hugoniot derivative}
\end{align}
Thus, by \eqref{eq:F_T} and \eqref{eq:Hugoniot derivative} we end up with the following expression for $dT/d\alpha$:
\begin{equation}
   \frac{1}{T}\frac{dT}{d\alpha} = \frac{\ds 4 +  2\frac{\Ti}{T } + \frac{p_{0}}{p}\frac{(1 + \alpha)(2 - \alpha)}{\alpha(1 - \alpha)}  + 2\frac{p}{p_{0}}\frac{T_0(1 + \alpha_{0})}{T\alpha(1 -  \alpha^{2})}}
{\ds (1 + \alpha)\frac{p_0}{p}\left(\frac{3}{2} + \frac{\Ti}{T}\right)  + (1 + \alpha_0)\frac{p}{p_0}\frac{T_0}{T}\left(\frac{5}{2} + \frac{\Ti}{T}\right) - 4(1 + \alpha )}. \label{eq:dT/dalpha:p}
\end{equation}
The numerator of the right-hand side in \eqref{eq:dT/dalpha:p} is strictly positive for every $0 < \alpha < 1$ and $T > 0$; the denominator is strictly positive in a neighbourhood of $\alpha = \alpha_0$ by \eqref{eq:Phipositive}. We conclude that $T(\alpha)$ is a strictly increasing function in this neighbourhood. 

\item Let $(\alpha_-,\alpha_+)$ denote the largest (open) interval where the function $T(\alpha)$ is defined. We claim that 
\[
\Phi\left(\alpha,T(\alpha)\right)>0 \quad \hbox{ for } \alpha\in(\alpha_-,\alpha_+).
\] 
Indeed, if $\alpha \in(\alpha_0, \alpha_+)$ then the claim follows by \eqref{eq:Phipositive}. If $\alpha \in (\alpha_-,\alpha_0)$, clearly $\Phi\left(\alpha,T(\alpha)\right) > 0$ in a small neighbourhood of $\alpha_0.$ Suppose that there is $\alpha_1$ such that $\alpha_- < \alpha_1 < \alpha_0$ with $\Phi\left(\alpha_1,T(\alpha_1)\right) =  0$ and $\Phi\left(\alpha,T(\alpha)\right) > 0$ for $\alpha\in(\alpha_1,\alpha_0)$. Since $T(\alpha)$ is increasing in $(\alpha_1, \alpha)$ by \eqref{eq:dT/dalpha:p}, we have $T_1 := T(\alpha_1) < T_0$ and $F(\alpha_1,T_1) = 0.$ Then \eqref{eq:Phipositive} yields $\Phi(\alpha_1,T_1) > 0$, which is a contradiction.

\item Next, we claim 
\[
\alpha_- = 0\quad\hbox{ and } \quad \alpha_+ = 1.
\] 
Note that, since $T = T(\alpha)$ is strictly increasing in $(\alpha_-,\alpha_+)$, then both limits $T_{\pm} = \lim_{\alpha \to \alpha_{\pm}}T(\alpha)$ exist. Assume by contradiction $\alpha_+<1$. If $T_+ < \infty,$ then $F(\alpha_+,T_+) = 0$ by continuity and, since $\alpha_+>\alpha_0$, we have $F_T(\alpha_+T_+) \neq  0$ by \eqref{eq:Phipositive}. But then the function $T(\alpha)$ can be extended beyond $\alpha_+,$ which is a contradiction. If $T_+ = \infty,$ we find by \eqref{eq:F} that $F\left(\alpha,T(\alpha)\right) \to -\infty$ as $\alpha \to \alpha_+$ which is impossible. We conclude that $\alpha_+ = 1;$ by the same argument, we have $\alpha_- = 0.$ 

\item At last, let us denote $T_{\infty} = \lim_{\alpha \to 1}T(\alpha).$ If $T_{\infty} < \infty,$  by \eqref{eq:F} we find $F\left(\alpha,T(\alpha)\right) \to \infty$ as $\alpha \to 1$ which is impossible; then $\lim_{\alpha \to 1}T(\alpha) = \infty.$ In the same way we conclude $\lim_{\alpha \to 0}T(\alpha) = 0.$ 
\end{enumerate}
The proposition is completely proved.
\end{proof}
\par
Since the function $T=T(\alpha)$ introduced in Proposition \ref{eq:Hugoniot curve alpha T} is strictly increasing, then it is invertible and defines a function $\alpha=\alpha(T)$. In the following proposition we precise the asymptotic behavior of $\alpha(T)$ when $T$ is close either to $0$ or $\infty$.

\begin{proposition}[Asymptotics]\label{prop:asymptotic}
On the Hugoniot locus \eqref{eq:Hugoniot-locus}, if $ T\to 0$ then $\alpha \to 0$ and more precisely
$$
	\alpha \sim  \frac{2\alpha_0}{\sqrt{1 - \alpha_0}} \left[1  + \frac{\Ti\alpha_0}{2T_0(1 + \alpha_0)}\right]^{\frac{1}{2}}\left(\frac{T}{T_0}\right)^{\!\frac{3}{4}}e^{- \frac{\Ti}{2T } + \frac{\Ti}{2T_0}}\quad \hbox{ as } T\to0.
$$
Furthermore, if $ T\to \infty$ then $\alpha \to 1$ and we have 
\begin{equation}\label{e:asymptotics_th-part-1}
	1 - \alpha \sim 4\frac{1 - \alpha_0}{\alpha_0^2} \left(\frac{T}{T_0}\right)^{-\frac{3}{2}}e^{-\frac{\Ti}{T_0}} \quad \hbox{ as } T\to\infty.
\end{equation}
\end{proposition}
\begin{proof}
By \eqref{eq:T_+/T_-prop} we see that $T \to 0$ if $ \alpha \to 0.$ Suppose that $\alpha^2T^{-\frac{5}{2}} e^{\frac{\Ti}{T}} = O(1)$, i.e. that it tends to a finite nonzero value for $T\to0$. Then, for $T\to0$, the left-hand side of \eqref{eq:T_+/T_-prop} would be $\sim 0$ while the right-hand side would be $O(1)$, a contradiction. 

Then, we set $\alpha \sim A\left(\frac{T}{T_0}\right)^{\!\mu}e^{- \frac{\Ti}{2T }},$ for some $A \geq 0.$ Around the point $(\alpha,T)=(0,0)$ we have
\begin{align*}
\lefteqn{\frac{T}{T_0} \left[\frac{8}{3}(1 + \alpha ) + \frac{2}{3}\left(\frac{1 - \alpha_0^2}{1 - \alpha }\right)\left(\frac{\alpha }{\alpha_0}\right)^{\!2}\left(\frac{T_0}{T }\right)^{\frac{5}{2}}
e^{-\frac{\Ti}{T_0} + \frac{\Ti}{T }} +  \frac{4}{3}\frac{\Ti}{T }\alpha \right]}
\\
 &\sim \frac{8}{3}\frac{T}{T_0} + \frac{2}{3}\left(\frac{1 - \alpha_0^2}{\alpha_0^2}\right)A^2\left(\frac{T }{T_0}\right)^{\! 2\mu - \frac{3}{2}} e^{-\frac{\Ti}{T_0}}\sim \frac{8}{3}(1 + \alpha_0) +  \frac{2}{3}\left(\frac{\alpha_0^2}{1 - \alpha_0}\right)A^{-2}\left(\frac{T }{T_0}\right)^{\! -2\mu + \frac{5}{2}}e^{\frac{\Ti}{T_0}} +  \frac{4}{3}\frac{\Ti}{T_0}\alpha_0.
\end{align*}
If $2\mu -\frac{3}{2} = 1,$ then $\mu = \frac{5}{4},$ which is impossible by the first part of the proof. If $2\mu -\frac{3}{2} = -2\mu + \frac{5}{2},$ then $\mu = 1,$ which is also impossible. Then $2\mu -\frac{3}{2} = 0$ and hence
$$
	\mu = \frac{3}{4} \quad \text{\rm and} \quad A = \left[\frac{4(1 + \alpha_0) + \frac{2\Ti}{T_0}\alpha_0}{\frac{1 - \alpha_0^2}{\alpha_0^2}}\right]^{\frac{1}{2}}e^{\frac{\Ti}{2T_0}} =  \frac{2\alpha_0}{\sqrt{1 - \alpha_0}} \left[1  + \frac{\Ti\alpha_0}{T_0(1 + \alpha_0)}\right]^{\frac{1}{2}}e^{\frac{\Ti}{2T_0}} .
$$
\par
On the other hand, if $\alpha \to 1$ then $ T \to \infty.$  Suppose that $(1 - \alpha)T^{\frac{5}{2}} = O(1).$ Then we have
\begin{align*}
  & \frac{T}{T_0} \left[\frac{8}{3}(1 + \alpha ) + \frac{2}{3}\left(\frac{1 - \alpha_0^2}{1 - \alpha }\right)\left(\frac{\alpha }{\alpha_0}\right)^{\!2}\left(\frac{T_0}{T }\right)^{\frac{5}{2}}
e^{-\frac{\Ti}{T_0} + \frac{\Ti}{T }} +  \frac{4}{3}\frac{\Ti}{T }\alpha \right]
 >  O(1)T, \\
 & \left[\frac{8}{3}(1 + \alpha_0) +  \frac{2}{3}\left(\frac{1 - \alpha ^2}{1 - \alpha_0}\right)\left(\frac{\alpha_0}{\alpha }\right)^{\!2}\left(\frac{T }{T_0}\right)^{\frac{5}{2}}
e^{-\frac{\Ti}{T } + \frac{\Ti}{T_0}} +  \frac{4}{3}\frac{\Ti}{T_0}\alpha_0\right]
 \sim  O(1),
\end{align*}
which is impossible. In this case, by setting  $1 - \alpha \sim B\left(\frac{T}{T_0}\right)^{\!- \mu}$ for $\mu >0,$  we have around $(1,\infty)$
\begin{align*}
\lefteqn{\frac{T}{T_0} \left[\frac{8}{3}(1 + \alpha ) + \frac{2}{3}\left(\frac{1 - \alpha_0^2}{1 - \alpha }\right)\left(\frac{\alpha }{\alpha_0}\right)^{\!2}\left(\frac{T_0}{T }\right)^{\!\frac{5}{2}}
e^{-\frac{\Ti}{T_0} + \frac{\Ti}{T }} +  \frac{4}{3}\frac{\Ti}{T }\alpha \right]}
\\
 &\sim \frac{16}{3}\frac{T}{T_0} + \frac{2}{3}\left(\frac{1 - \alpha_0^2}{\alpha_0^2}\right)B^{-1}\left(\frac{T }{T_0}\right)^{\! \mu - \frac{3}{2}} e^{-\frac{\Ti}{T_0}}
 \sim  \frac{8}{3}(1 + \alpha_0) +  \frac{4}{3}\left(\frac{\alpha_0^2}{1 - \alpha_0}\right)B\left(\frac{T }{T_0}\right)^{\! -\mu + \frac{5}{2}}e^{\frac{\Ti}{T_0}}  + \frac{4}{3}\frac{\Ti}{T_0}\alpha_0.
\end{align*}

If $\mu - \frac{3}{2} =  - \mu + \frac{5}{2},$ then $\mu = 2$ which is impossible.  Then $- \mu + \frac{5}{2} = 1$ and so we deduce both $\mu = \frac{3}{2}$ and $B = 4\left(\frac{1 - \alpha_0}{\alpha_0^2}\right)e^{-\frac{\Ti}{T_0}}$.
\end{proof}

\begin{nb}\label{rem:asymptotics}
Let us consider the asymptotics for $\frac{T}{T_0} \to \infty$, $\frac{\Ti}{T} = O(1)$ and $\alpha_0 \to 0$, $T_0 \to 0$. Exchanging the roles of $T$ and $T_0$ in the above argument, we have the asymptotic formula
$$
	\alpha_0 \sim  \left[\frac{4(1 + \alpha) + \frac{2\Ti}{T}\alpha}{\frac{1 - \alpha^2}{\alpha^2}}\right]^{\frac{1}{2}}\left(\frac{T_0}{T}\right)^{\!\frac{3}{4}}e^{- \frac{\Ti}{2T_0 } + \frac{\Ti}{2T}} \qquad \text{ for $\frac{T}{T_0} \to \infty$, $\frac{\Ti}{T} = O(1)$, $\alpha_0 \to 0$, $T_0 \to 0$}.
$$
\end{nb}
%
%
\subsection{The variation of $p$ and $v$ along the Hugoniot locus}
In this section we exploit Proposition \ref{eq:Hugoniot curve alpha T} to compute the variation of $p$ and $v$ along the thermodynamic part $\left(\alpha, T(\alpha)\right)$ of the Hugoniot locus. We prove that $p$ increases and $v$ decreases as $\alpha$ increases. 

\begin{proposition}[$p$ and $v$ as functions of $\alpha$]\label{p:dp/dalpha > 0} Let $T=T(\alpha)$ be the function introduced in Proposition \ref{eq:Hugoniot curve alpha T}. Then
$$
    \frac{d}{d \alpha} p\left(\alpha, T(\alpha)\right)> 0 \quad \hbox{ and }\quad \frac{d}{d \alpha} v\left(\alpha, T(\alpha)\right) < 0\qquad \text{ for } \alpha \geq \alpha_0.
$$
\end{proposition}
\begin{proof}
By \eqref{eq:p-alphaT} and \eqref{eq:prhoT} we deduce
\begin{align}
\label{eq:p(alpha)}
  \log p & = \log (1 - \alpha^2) - 2\log \alpha + \frac{5}{2} \log T - \frac{\Ti}{T} + \text{const.}
\\
  \log v &
  = -\log (1 - \alpha) + 2\log \alpha - \frac{3}{2}\log T + \frac{\Ti}{T} + \text{const.}
\label{eq:valpha}
\end{align}
In the following of the proof, we work on the thermodynamic part $\left(\alpha, T(\alpha)\right)$ of the Hugoniot locus; then $T=T(\alpha)$ and,  with a slight abuse of notation, we denote $p(\alpha):= p\left(\alpha, T(\alpha)\right)$ and $v(\alpha):= v\left(\alpha, T(\alpha)\right).$ 

\begin{enumerate}[{\em (i)}]

\item First, we consider the variation of $p$. By \eqref{eq:p(alpha)} we compute
\begin{equation}\label{eq:dpH}
  \frac{d}{d\alpha} \log p = - \frac{2}{\alpha(1 - \alpha^2)} + \left(\frac{5}{2} + \frac{\Ti}{T}\right)\frac{1}{T}\frac{dT}{d\alpha},
\end{equation}
and then, by \eqref{eq:dT/dalpha:p},
\begin{align}
\frac{1}{2p}\alpha(1 - \alpha^2)\frac{dp}{d\alpha}
&= -1 + \frac{\frac{1}{2}\alpha(1 - \alpha^2)\left[1 + \left(\frac{p_{0}}{p}\right)\frac{(1 + \alpha)(2 - \alpha)}{4\alpha(1 - \alpha)} +  \frac{\Ti}{2T } +\left(\frac{p}{p_{0}}\right)\frac{T_0(1 + \alpha_{0})}{2T\alpha(1 -  \alpha^{2})}\right]\left(\frac{5}{2} + \frac{\Ti}{T}\right)}
  {\frac{1}{4}(1 + \alpha)\left(\frac{p_0}{p}\right)\left(\frac{3}{2} + \frac{\Ti}{T}\right)  +  \frac{1}{4}(1 + \alpha_0)\left(\frac{p}{p_0}\right)\frac{T_0}{T}\left(\frac{5}{2} + \frac{\Ti}{T}\right) - (1 + \alpha ) }
\nonumber
\\
&=:\frac{N(\alpha)}{D(\alpha)}.
\label{eq:dpdalpha}
\end{align}
The denominator $D(\alpha)$ is positive by the proof of Proposition \ref{eq:Hugoniot curve alpha T} while 
\begin{align}
N(\alpha) & =\frac{1}{2}\alpha(1 - \alpha^2)\left(1 + \frac{\Ti}{2T}\right)\left(\frac{5}{2} + \frac{\Ti}{T}\right)+ \frac{p_{0}}{p}\frac{(1 + \alpha)^2(2 - \alpha)}{8}\left(\frac{5}{2} + \frac{\Ti}{T}\right) 
\nonumber
\\
& \quad - \frac{1}{4}(1 + \alpha)\frac{p_0}{p}\left(\frac{3}{2} + \frac{\Ti}{T}\right) + (1 + \alpha ).\label{eq:numerator-p}
\end{align}
Assume $p \geq p_0.$ For the last two summands in \eqref{eq:numerator-p} we have
$$
- \frac{1}{4}(1 + \alpha)\frac{p_0}{p}\left(\frac{3}{2} + \frac{\Ti}{T}\right) + (1 + \alpha )  \ge - \frac{1}{4}(1 + \alpha)\frac{p_0}{p}\frac{3}{2} + (1 + \alpha)
\geq \frac{5}{8}(1 + \alpha).
$$
Then, we deduce
\begin{equation}\label{eq:Nalpha}
N(\alpha)\ge \frac{\Ti}{2T}\alpha(1 - \alpha^2)\left[1 + \frac{1}{4}\frac{p_{0}}{p}\right] > 0, \quad \hbox{ for } p\ge p_0.
\end{equation}
Estimate \eqref{eq:Nalpha} holds in particular for $p=p_0$ and then shows that $p'(\alpha_0) > 0.$ Then $p'(\alpha)> 0$ and $p(\alpha) > p_0$ for $\alpha\in[\alpha_0 , \alpha_1]$, where $\alpha_1$ is close to $\alpha_0$ Now, assume that there exists $\alpha_*>\alpha_1$ such that $p'(\alpha_*) = 0$ and $p(\alpha) \geq p_0$ for $\alpha\in[\alpha_0 ,\alpha_{\ast}]$; then \eqref{eq:Nalpha} also shows that $p'(\alpha_*) > 0,$ which is a contradiction.  Thus we conclude that $p'(\alpha) > 0$ for all $\alpha \geq \alpha_0.$

\item Now, we consider the variation of $v$.  The result is a counterpart of what proved above but it is not directly deduced from it. 
By \eqref{eq:valpha} and \eqref{eq:dT/dalpha:p} it follows
\[
  \frac{1}{v}\frac{dv}{d\alpha}
 \ =\  \frac{2 - \alpha}{\alpha(1 - \alpha)}
  - \frac{\left[1  +  \frac{\Ti}{2T } + \left(\frac{p_{0}}{p}\right)\frac{(1 + \alpha)(2 - \alpha)}{4\alpha(1 - \alpha)} +\left(\frac{p}{p_{0}}\right)\frac{T_0(1 + \alpha_{0})}{2T\alpha(1 -  \alpha^{2})}\right]\left(\frac{3}{2} + \frac{\Ti}{T}\right)}
  {\frac{1}{4}(1 + \alpha)\left(\frac{p_0}{p}\right)\left(\frac{3}{2} + \frac{\Ti}{T}\right)  +  \frac{1}{4}(1 + \alpha_0)\left(\frac{p}{p_0}\right)\frac{T_0}{T}\left(\frac{5}{2} + \frac{\Ti}{T}\right) - (1 + \alpha ) }
=:\frac{N_1(\alpha)}{D(\alpha)}.
\]
As above, $D(\alpha)$ is positive, while some computations lead to 
\begin{align}
N_1(\alpha)    &= \frac{(2 - \alpha)(1 + \alpha)}{4\alpha(1 - \alpha)} \left[\frac{T_0(1 + \alpha_0)}{T(1 + \alpha)} - 4\right] + \frac{1}{4}\left[\frac{T_0(1 + \alpha_0)}{T(1 + \alpha)}\frac{p}{p_0} - 4\left(1  +  \frac{\Ti}{2T }\right)\right]\left(\frac{3}{2} + \frac{\Ti}{T}\right).\label{eq:last}
\end{align}
We claim that the term in the second bracket in \eqref{eq:last} is negative for $\alpha\ge\alpha_0$, i.e.,
 \begin{equation}\label{e:4-3}
 \frac{T_0(1 + \alpha_0)}{T(1 + \alpha)}\frac{p}{p_0} - 4\left(1  +  \frac{\Ti}{2T }\right) = \frac{v_0}{v} - 4\left(1  +  \frac{\Ti}{2T }\right)\le0, \quad \hbox{ for } \alpha\ge\alpha_0.
 \end{equation}
This would conclude the proof of the lemma: indeed, the first summand in \eqref{eq:last} is strictly negative because $T(\alpha)$ is a strictly decreasing function in $(0,1)$ by Proposition \ref{eq:Hugoniot curve alpha T}.

To prove \eqref{e:4-3}, let us denote $\hat{\Pi} = \frac{p}{p_0}$. By \eqref{eq:quadratic}, we have that $\hat{\Pi}$ satisfies the quadratic equation
$$
	\frac{T_0(1 + \alpha_0)}{T(1 + \alpha)}\hat{\Pi}^{2} -  4\Delta\hat{\Pi} - 1 = 0,
$$
where $\Delta$ is defined in \eqref{e:DeltaD}. Since $\hat{\Pi} \geq 1$ for $\alpha\ge\alpha_0$ by the previous step {\em (i)}, we deduce
$$
\frac{T_0(1 + \alpha_0)}{T(1 + \alpha)}\hat{\Pi} = 2\Delta + \sqrt{4\Delta^2 +  \frac{T_0(1 + \alpha_0)}{T(1 + \alpha)}}
$$
and then 
\begin{equation}\label{e:5-3}
 \frac{T_0(1 + \alpha)}{T(1 + \alpha)}\left(\frac{p}{p_0}\right) - 4\left(1 + \frac{\Ti}{2T}\right)
= 2\Delta + \sqrt{4\Delta^2 + \frac{T_0(1 + \alpha_0)}{T(1 + \alpha)}} - 4\left(1 + \frac{\Ti}{2T}\right).
\end{equation}
We have that $\Delta \geq 0$ for $\alpha \in(0,1)$ because the function $T(\alpha)$ is strictly decreasing in $(0,1)$ by Proposition \ref{eq:Hugoniot curve alpha T}.
If $\Delta = 0,$ then \eqref{e:5-3} is strictly negative. Suppose that \eqref{e:5-3} vanishes for some $\Delta > 0.$ Then
$$
   \sqrt{4\Delta^2 + \frac{T_0(1 + \alpha_0)}{T(1 + \alpha)}}
= 4\left(1 + \frac{\Ti}{2T}\right) - 2\Delta
$$
and so
\[
	\frac{T_0(1 + \alpha_0)}{T(1 + \alpha)}
= 
    16\left(1 + \frac{\Ti}{2T}\right)\left[\frac{\Ti}{2T} +  \frac{T_0(1 + \alpha_0)}{T(1 + \alpha)} - \frac{\Ti(\alpha - \alpha_0)}{2T(1 + \alpha)}\right].
\]
However, the above equation can never be satisfied, since $\frac{\Ti}{2T} - \frac{\Ti(\alpha - \alpha_0)}{2T(1 + \alpha)} > 0$ for every $\alpha\in(0,1)$. Then claim \eqref{e:4-3} holds and the lemma is proved.
\end{enumerate}
\qedhere
\end{proof}

\subsection{The kinetic part of the Hugoniot locus}
In this subsection we focus on the variation of the velocity $u$ along the Hugoniot locus \eqref{eq:Hugoniot-locus} of $(\alpha_0,u_0,T_0)$; these results are used in the forthcoming paper \cite{Asakura-Corli_reflected} to study the shock tube problem. By \eqref{eq:prhoT}, equation $\eqref{eq:Hugoniot-locus}_1$ can be written as
\begin{eqnarray*}
  (u - u_0)^2 = - p_0 v_0 \left(\frac{p}{p_0} - 1\right) \left(\frac{v}{v_0} - 1\right) 
= - p_0v_0 \left[\frac{T(1 + \alpha)}{T_0(1 + \alpha_0)} - \frac{p}{p_0} - \frac{v}{v_0} + 1 \right].
\end{eqnarray*}
Therefore, by \eqref{eq:prhoT}, a point of the Hugoniot locus of $(\alpha_0,u_0,T_0)$ lies in the intersection of the set
\begin{equation}\label{eq:V alpha T curve}
G(\alpha,u,T):= \left(\frac{p}{p_0} - 1\right) \left(\frac{v}{v_0} - 1\right)  + \frac{ (u - u_0)^2}{a^2T_0(1 + \alpha_0)} =0
\end{equation}
with $\eqref{eq:Hugoniot-locus}_2$. For simplicity, we drop the dependence of $G$ on $(\alpha_0,u_0,T_0)$; of course, we have $G(\alpha_0,u_0,T_0)=0$. Since $u$ is fixed in the following, we understand the set $G(\alpha,u,T)$ as a set in the $(\alpha,T)$-plane, see Figure \ref{fig:G_graph}.

\begin{proposition}[The kinetic part of the Hugoniot locus]\label{prop:piece of curve}
Fix $(\alpha_0,u_0,T_0)$ and $u\ne u_0$; then, the set $G(\alpha,u,T) = 0$ is constituted by the graphs of two continuous functions $T_-(\alpha)<T_+(\alpha)$, for $\alpha\in(0,1)$, satisfying $T_-(\alpha_0)<T_0<T_+(\alpha_0).$ The function $T_+(\alpha)$ is differentiable and its graph is located in the region $p(\alpha,T) > p_0.$
\par
Moreover, we have $\lim_{\alpha \to 1}T_{\pm}(\alpha)=\infty$; for $\alpha \to 1$, on the set $G(\alpha,u,T)=0$ we have
\begin{equation}\label{e:asymptotics-kinetic-part}
  1 - \alpha \sim \frac{1 - \alpha_0}{\alpha_0^2}e^{- \frac{\Ti}{T_0}}\left(\frac{T}{T_0}\right)^{\! -\frac{3}{2}}\ \hbox{along $T_+$}\quad \hbox{ and }\quad  1 - \alpha \sim \frac{1 - \alpha_0^2}{2\alpha_0^2}e^{- \frac{\Ti}{T_0}}\left(\frac{T}{T_0}\right)^{\! -\frac{5}{2}}\ \hbox{along $T_-$.}
\end{equation}
Near $\alpha = 1,$ both curves are located under the thermodynamic part $T=T(\alpha)$ of the Hugoniot locus of $(\alpha_0,T_0)$. 
\end{proposition}
\begin{figure}[htb]
  \centering
  \begin{tabular}{c}
\includegraphics[width =.65\linewidth]{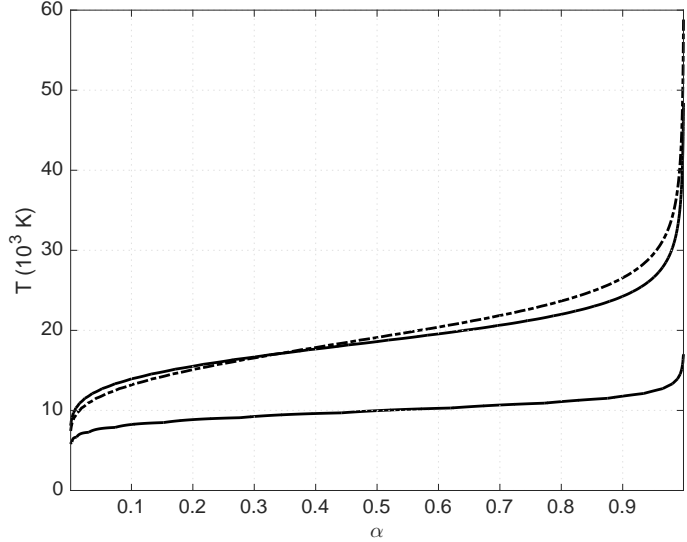}
\end{tabular}
\caption{Solid line: the set $G(\alpha,u, T)=0$ in the plane $(\alpha,T)$ for fixed $u$; dash-point line: the thermodynamic part of the Hugoniot locus. Here 
$T_0 = 300{\rm K}$, $\alpha_0 = 3.5929 \times 10^{-114}$ and $u_0 = 0$; we used $u= 38000\, {\rm m\, s}{}^{-1}$ instead of the value $u=8100\, {\rm m\, s}{}^{-1}$ of \cite{Fukuda-Okasaka-Fujimoto}, to make the intersection clearer.}\label{fig:G_graph}
\end{figure}
\begin{proof}
By \eqref{eq:p_+/p_-} and \eqref{eq:p-alphaT} we can eliminate $p$ from the expression of $G$ and write
\begin{align}
\lefteqn{G(\alpha, u, T)
   = \frac{p}{p_0} + \frac{v}{v_0} -  \frac{T(1 + \alpha)}{T_0(1 + \alpha_0)} - 1  -  \frac{ (u - u_0)^2}{a^2T_0(1 + \alpha_0)}}
\nonumber
\\
 &=\frac{\alpha_0^2(1 - \alpha^2)}{\alpha^2(1 - \alpha_0^2)}\left(\frac{T}{T_0}\right)^{\!\frac{5}{2}}e^{-\frac{\Ti}{T} + \frac{\Ti}{T_0}} + \frac{\alpha^2(1 - \alpha_0)}{\alpha_0^2(1 - \alpha)}\left(\frac{T_0}{T}\right)^{\!\frac{3}{2}}e^{-\frac{\Ti}{T_0} + \frac{\Ti}{T}}-  \frac{T(1 + \alpha)}{T_0(1 + \alpha_0)} - 1 - \frac{ (u - u_0)^2}{a^2T_0(1 + \alpha_0)}.
 \label{e:GGG}
\end{align}
The function $T\mapsto(T/T_0)^{\!\frac{5}{2}}e^{-\frac{\Ti}{T} + \frac{\Ti}{T_0}}$ is strictly increasing and valued in $[0,\infty)$. Then, for every $\alpha > 0$ there exists a unique $T_{\ast} = T_{\ast}(\alpha)$ such that, by \eqref{eq:p-alphaT},
$$
\frac{p_{\ast}}{p_0} = \left(\frac{1 - \alpha^2}{1 - \alpha_0^2}\right)\left(\frac{\alpha_0^2}{\alpha^2}\right)\left(\frac{T_{\ast}}{T_0}\right)^{\!\frac{5}{2}}e^{-\frac{\Ti}{T_{\ast}} + \frac{\Ti}{T_0}} = 1,
$$
for $p_*=p\left(\alpha,T_*(\alpha)\right)$; moreover, $T_\ast(\alpha_0)=T_0.$ Therefore 
$
G\left(\alpha,u, T_{\ast}(\alpha)\right) = - \frac{ (u - u_0)^2}{a^2T_0(1 + \alpha_0)}  < 0.
$
\par
For every $\alpha$ fixed we have that $G(\alpha,u, T) \to \infty$ for both $T \to 0$ and $T\to\infty$.
We conclude that there are at least two values $T_\pm = T_\pm(\alpha)$, with $T_- < T_*<T_+$, such that $G\left(\alpha,u, T_\pm(\alpha)\right)= 0$. Note that $p > p_0$ along $T = T_+(\alpha)$ and $p < p_0$ along $T = T_-(\alpha).$
Moreover, by a direct computation we find
\begin{align*}
	T_0\frac{\partial G}{\partial T}(\alpha,u, T)
	&=\frac{\alpha_0^{2}(1 - \alpha^{2})}{\alpha^{2}(1 - \alpha_0^{2})}\left(\frac{5}{2} + \frac{\Ti}{T}\right)\left(\frac{T}{T_0}\right)^{\!\frac{3}{2}}e^{-\frac{\Ti}{T} + \frac{\Ti}{T_0}} \\
	& 	\quad - \frac{\alpha^{2}(1 - \alpha_0)}{\alpha_0^{2}(1 - \alpha)}\left(\frac{3}{2} + \frac{\Ti}{T}\right)\left(\frac{T_0}{T}\right)^{\!\frac{5}{2}}e^{-\frac{\Ti}{T_0} + \frac{\Ti}{T}} - \frac{1 + \alpha}{1 + \alpha_0} \\
T_0^2\frac{\partial^2 G}{\partial T^2}(\alpha,u, T)  
        &= \frac{\alpha_0^{2}(1 - \alpha^{2})}{\alpha^{2}(1 - \alpha_0^{2})}\left[\left(\frac52 + \frac{\Ti}{T}\right)\left(\frac32 + \frac{\Ti}{T}\right) - \frac{\Ti}{T}\right]\left(\frac{T}{T_0}\right)^{\!\frac{1}{2}}e^{-\frac{\Ti}{T} + \frac{\Ti}{T_0}} \\
        &\quad+ \frac{\alpha^{2}(1 - \alpha_0)}{\alpha_0^{2}(1 - \alpha)}\left[\left(\frac32 + \frac{\Ti}{T}\right)\left(\frac52 + \frac{\Ti}{T}\right)  + \frac{\Ti}{T}\right]\left(\frac{T_0}{T}\right)^{\!\frac{7}{2}}e^{-\frac{\Ti}{T_0} + \frac{\Ti}{T}} > 0.
\end{align*} 
Then, for every $\alpha\in(0,1)$ the equation $G(\alpha,u, T)=0$ has exactly the two zeroes $T_\pm(\alpha)$.
\par
 We have, by \eqref{eq:spec-vol-ion}, \eqref{eq:p-alphaT} and \eqref{eq:V alpha T curve},
\begin{align}
\frac{\partial G}{\partial T}(\alpha,u, T)
	&=
\frac{1}{T}\left[\left(\frac{5}{2} + \frac{\Ti}{T}\right)\frac{p}{p_0} - \left(\frac{3}{2} + \frac{\Ti}{T}\right)\frac{v}{v_0} - \frac{T(1 + \alpha)}{T_0(1 + \alpha_0)}\right]\nonumber
\\
  & = \frac{1}{T}\left[\left(\frac{3}{2} + \frac{\Ti}{T}\right)\left(\frac{p}{p_0} - \frac{v}{v_0}\right) + \left(1 - \frac{v}{v_0}\right)   + \frac{ (u - u_0)^2}{a^2T_0(1 + \alpha_0)}\right].\label{e:GGT}
\end{align}
On the graph of $T = T_+(\alpha)$ we proved that $p > p_0$ and then $\eqref{eq:Hugoniot-locus}_1$ shows that $v < v_0$; as a consequence, \eqref{e:GGT} yields $\frac{\partial G}{\partial T}(\alpha,u, T) > 0$ there. Thus we conclude that $T_+(\alpha)$ is differentiable.
%

At last, we prove the asymptotic behavior of the set $G(\alpha,u,T)=0$ for $\alpha$ close to $1$. It is easy to see from \eqref{e:GGG} that on this set we have $T \to \infty$ if $\alpha\to1$. We set
$
   1 - \alpha \sim B\left(\frac{T}{T_0}\right)^{\! -\mu}
$  
for $\mu>0$. By \eqref{e:GGG} we have
\begin{eqnarray*}
\frac{2B\alpha_0^2}{1 - \alpha_0^2}\left(\frac{T}{T_0}\right)^{\!\frac{5}{2} - \mu}e^{ \frac{\Ti}{T_0}}   +  \frac{1 - \alpha_0}{\alpha_0^2B}\left(\frac{T_0}{T}\right)^{\!\frac{3}{2} - \mu}e^{-\frac{\Ti}{T_0}}
  &\sim &   \frac{2T}{T_0(1 + \alpha_0)} + 1 + \frac{ (u - u_0)^2}{a^2T_0(1 + \alpha_0)} .
\end{eqnarray*} 
If $T^{\frac{5}{2} - \mu}$ and $T^{\mu - \frac{3}{2}}$ are equally large as $T \to \infty,$ we have $\mu = 2.$ Then both these terms are $O(1)T^{\frac{1}{2}},$ which is a contradiction. In the case $\frac{5}{2} - \mu = 1,$ we have $\mu = \frac{3}{2}$ and $B = \frac{1 - \alpha_0}{\alpha_0^2}e^{- \frac{\Ti}{T_0}} .$ Otherwise 
$\mu = \frac{5}{2}$ and  $B = \frac{1 - \alpha_0^2}{2\alpha_0^2}e^{- \frac{\Ti}{T_0}} .$ By Proposition \ref{prop:asymptotic}, we conclude that the graphs of both $T_\pm$ are located under the thermodynamic part $T=T(\alpha)$ of the Hugoniot locus of $(\alpha_0,T_0)$ near $\alpha = 1.$
\end{proof}
\begin{lemma}[Intersection of Hugoniot loci]\label{lem:intersection}
Fix $(\alpha_0,u_0, T_0)$ and $u\ne u_0$. 
Then, in the region $\alpha_0 < \alpha < 1$ there is at least one point of intersection between the set $G(\alpha, u, T) = 0$ and the thermodynamic part $\eqref{eq:Hugoniot-locus}_2$ of the Hugoniot locus of $(\alpha_0,T_0)$. The same result holds also in the region $0<\alpha<\alpha_0$ if $(u - u_0)^2 \leq 2a^2\Ti\alpha_0.$
\end{lemma}
\begin{proof}
We refer to Figure \ref{fig:G_graph}. By the proof of Proposition \ref{prop:piece of curve} we know that $G(\alpha_0,u, T_0) < 0$, while by Proposition \ref{eq:Hugoniot curve alpha T} we have $T(\alpha_0)=0$. Then, in a neighborhood of $(\alpha_0,T_0),$ the curve $T=T(\alpha$) is located between the two curves $T = T_-(\alpha)$ and $T = T_+(\alpha)$. 

In the region $\alpha_0<\alpha<1$, it follows by \eqref{e:asymptotics-kinetic-part} and \eqref{e:asymptotics_th-part-1} that there is at least one point of intersection between the curves $T = T_+(\alpha)$ and $T=T(\alpha)$.

On the other hand, Proposition \ref{prop:asymptotic} and \eqref{e:GGG} imply 
$
  G(\alpha, u, T) \sim 3 + \frac{1}{T_0(1 + \alpha_0)}\left[2\Ti\alpha_0 - \frac{(u - u_0)^2}{a^2}\right]
$
as $\alpha, T \to 0$ along the thermodynamic part of the Hugoniot locus. Then, the set $G(\alpha,u,T) = 0$ and the graph of the function $T(\alpha)$ intersects in the region $0 < \alpha < \alpha_0$ provided $(u - u_0)^2 \leq 2a^2\Ti\alpha_0.$
\end{proof}



At last, from the results obtained in this section we conclude with the following main result.

\begin{theorem}[The Hugoniot locus]\label{thm:intersection}
Fix $(\alpha_0,u_0,T_0)$ and $u\ne u_0$. Then there exists a unique point $(\alpha, T)$ with $\alpha\in(\alpha_0,1)$, such that $(\alpha, u, T)$ belongs to the Hugoniot locus \eqref{eq:Hugoniot-locus} of $(\alpha_0,u_0,T_0)$.
 \end{theorem}
\begin{proof} We must prove that there is a unique point of intersection between the thermodynamic part $T=T(\alpha)$ of the Hugoniot locus of $(\alpha_0,T_0)$ and the set $G(\alpha, u, T) = 0$, where $u\ne u_0$ is fixed. To this aim, it is sufficient to prove that
\begin{equation}\label{e:ddalphaG}
     \frac{d}{d \alpha} G\left(\alpha,  u, T(\alpha)\right) > 0,
\end{equation} 
since then the point of intersection obtained in Lemma \ref{lem:intersection} is unique. 
 
By \eqref{e:GGG} and \eqref{e:GGT}, on the set $G(\alpha, u, T) = 0$ we have
\begin{eqnarray*}
\frac{d}{d\alpha} G\left(\alpha, u, T(\alpha)\right)
   &=& - \frac{2p}{\alpha(1 - \alpha^2)p_0} + \frac{(2 - \alpha)v}{\alpha(1 - \alpha)v_0} - \frac{T}{T_0(1 + \alpha_0)} \\
   & &+ \left[\left(\frac{5}{2} + \frac{\Ti}{T}\right)\frac{p}{p_0} - \left(\frac{3}{2} + \frac{\Ti}{T}\right)\frac{v}{v_0} - \frac{T(1 + \alpha)}{T_0(1 + \alpha_0)}\right] \frac{1}{T}\frac{dT}{d\alpha}
\end{eqnarray*}
and then by \eqref{eq:dT/dalpha:p} we deduce
\begin{align*}
\lefteqn{\alpha(1 - \alpha^2)\frac{dG}{d\alpha}\left(\alpha, u, T(\alpha)\right)} 
\\
  &= - \frac{2p}{p_0} + \frac{(2 - \alpha)(1 + \alpha)v}{v_0} - \alpha(1 - \alpha)\frac{T(1 + \alpha)}{T_0(1 + \alpha_0)}  + \left[\left(\frac{5}{2} + \frac{\Ti}{T}\right)\frac{p}{p_0} - \left(\frac{3}{2} + \frac{\Ti}{T}\right)\frac{v}{v_0} - \frac{T(1 + \alpha)}{T_0(1 + \alpha_0)}\right]\\
   &  \quad\times
   \frac{\frac{2p}{p_0} + \left(\frac{v}{v_0}\right)(1 + \alpha)(2 - \alpha) + 4\alpha(1 - \alpha)\left(1 +  \frac{\Ti}{2T }\right)\frac{T(1 + \alpha)}{T_0(1 +  \alpha_0)}}
  { \left(\frac{5}{2} + \frac{\Ti}{T}\right) \frac{p}{p_0} + \left(\frac{3}{2} + \frac{\Ti}{T}\right) \frac{v}{v_0} - \frac{4T(1 + \alpha)}{T_0(1 +  \alpha_0)}  } =: \frac{N_2(\alpha)}{D(\alpha)}.
 \end{align*}
As above, the denominator $D(\alpha)$ is positive by the proof of Proposition \ref{eq:Hugoniot curve alpha T}. 
After some algebraic computations we obtain
 \begin{align}
N_2(\alpha)= & \frac{T(1 + \alpha)}{T_0(1 + \alpha_0)}\left[6\left(\frac{p}{p_0}- 1\right) + 5\left(1 -  \frac{v}{v_0}\right)(2 - \alpha)(1 + \alpha)\right]\nonumber \\
   	& + 2\alpha(1 - \alpha)\frac{T(1 + \alpha)}{T_0(1 + \alpha_0)} \left[\left(\frac{3}{2} + \frac{\Ti}{T}\right)\left(\frac{5}{2} + \frac{\Ti}{T}\right) \left(\frac{p}{p_0} - \frac{v}{v_0}\right)  - \frac{\Ti}{T}
	\left\{\frac{T(1 + \alpha)}{T_0(1 + \alpha_0)} - 1\right\}\right]. 
\label{e:N2} \end{align}
By \eqref{eq:T_+/T_-}, the thermodynamic part of the Hugoniot locus can be written as
 $$
 	4T(1 + \alpha) + \frac{p_0}{p}T(1 + \alpha) + 2\Ti\alpha
	= 4T_0(1 + \alpha_0) + \frac{p}{p_0}T_0(1 + \alpha_0) + 2\Ti\alpha_0,
 $$ 
or, equivalently, as
 $$
 	\frac{p}{p_0} - \frac{v}{v_0} = 4\left[\frac{T(1 + \alpha)}{T_0(1 + \alpha_0)} - 1\right] + \frac{2\Ti(\alpha - \alpha_0)}{T_0(1 + \alpha_0)}.
 $$
Moreover, on the thermodynamic part of the Hugoniot locus we have $\frac{p}{p_0} \geq 1$ and $\frac{v}{v_0} \leq 1$ by Proposition \ref{p:dp/dalpha > 0}. This proves that $N_2(\alpha)>0$ for $\alpha\in(\alpha_0,1)$. Then, claim \eqref{e:ddalphaG} holds and the theorem is proved. 
 \end{proof}

\section{The variation of the entropy}\label{sec:entropy}

We define the {\em amount of entropy in the interval $(a,b)$} by
$$
    \mathcal{H}\left(t; [a,b]\right) = \int_a^b \rho(x,t)S(x,t)\,dx.
$$
We want to compute the variation of $\mathcal{H}$ when a (single) shock with constant velocity $s$ is present in the region under consideration, namely, $a<st<b$. We use the notation introduced at the beginning of Section \ref{sec:Hugoniot}, see \eqref{eq:Rankine-Hugoniot}.

\begin{lemma}\label{lem:entropy variation}
Assume that a shock front with constant velocity $s$ is passing in the interval $[a,b]$. Then $\rho_{+}(s - u_{+})=\rho_{-}(s - u_{-})$ and with $m = \rho_{\pm}(s - u_{\pm})$
we have
$$
  \frac{d}{dt} \mathcal{H}\left(t; [a,b]\right) = -m[S]_-^+ -  [\rho u S]_a^b.
$$
\end{lemma}
\begin{proof} By \eqref{eq:entropy conservation} we have
$$
   \frac{d}{dt} \mathcal{H}\left(t; [a,b]\right) +   [\rho u S]_a^b= -s[\rho S]_-^+ + [\rho u S]_-^+.
$$
By the Rankine-Hugoniot condition $\eqref{eq:RH}_1$ we have
$\rho_+(s - u_+) = \rho_-(s - u_-)$. Then
\begin{eqnarray*}
s[\rho S]_-^+ - [\rho u S]_-^+ &=& s(\rho_+S_+ - \rho_-S_-) - (\rho_+u_+S_+ - \rho_-u_-S_-)\\
  &=& \rho_+(s - u_+)S_+ - \rho_-(s - u_-)S_- = m[S]_-^+
\end{eqnarray*}
and the lemma is proved.
\end{proof}
We notice that by $\eqref{eq:RH}_2$ we have $m^2 = -\frac{p_+ - p_-}{v_+ - v_-},$ which shows that $m$ is the {\it Lagrangian shock speed\/}. 
\par
With the help of the previous lemma we now provide a result about the increase of entropy across a shock front. 

\begin{proposition}\label{prop:entropy variation}
Let $(p_-, u_-, S_-)$ and $(p_+,u_+,S_+)$ be connected by a shock wave and $v_{pp}(p_\pm,S_\pm)>0$. If $|p_+ - p_-|$ is sufficiently small, then the entropy increases across the shock front.
\end{proposition}
\begin{proof}
We begin by considering an asymptotic expression of $[S]_-^+ = S_+ - S_-$ for small $|p_+ - p_-|.$ By \eqref{eq:enthalpy}, we have
$$
    \left(\frac{\partial H}{\partial S}\right)_{\substack{p\\ \rule{1ex}{0ex}}} = T, \quad \left(\frac{\partial H}{\partial p}\right)_{\substack{S\\ \rule{1ex}{0ex}}} = v.
$$
Thus, on the one hand we have
\begin{align*}
  H_+ - H_- &=  
T_-(S_+ - S_-)  
+ v_-(p_+ - p_-)
     + \frac{1}{2} \left(\frac{\partial v}{\partial p_-}\right)_{\substack{S\\ \rule{1ex}{0ex}}}(p_+ - p_-)^2 + \frac{1}{6} \left(\frac{\partial^2 v}{\partial p_-^2}\right)_{\substack{S\\ \rule{1ex}{0ex}}}(p_+ - p_-)^3\\
  &\quad +  O(1)(S_+ - S_-)^2 + O(1)(S_+ - S_-)(p_+ - p_-) + O(1)(p_+ - p_-)^4.
\end{align*}
Here we used the shortcut $(\partial v/\partial p_-)_S=(\partial v/\partial p)(p_-,S_-)$. On the other hand, by $\eqref{eq:Hugoniot-locus}_2$ we have
\begin{align*}
  H_+ - H_- &=  \frac{1}{2}(v_+ + v_-)(p_+ - p_-)
   =  v_-(p_+ - p_-) + \frac{1}{2}(v_+ - v_-)(p_+ - p_-)\\
  & =  v_-(p_+ - p_-)
   +  \frac{1}{2} \left(\frac{\partial v}{\partial p_-}\right)_{\substack{S\\ \rule{1ex}{0ex}}}(p_+ - p_-)^2  +  \frac{1}{4} \left(\frac{\partial^2 v}{\partial p_-^2}\right)_{\substack{S\\ \rule{1ex}{0ex}}}(p_+ - p_-)^3\\
  &\quad  +  O(1)(S_+ - S_-)(p_+ - p_-) + O(1)(p_+- p_-)^4. 
\end{align*}
Combining the above two expressions and omitting some technical details, we conclude that
%
\begin{equation}\label{eq:entropy variation}
   T_-(S_+ - S_-) = \frac{1}{12} \left(\frac{\partial^2 v}{\partial p_-^2}\right)_{\substack{S\\ \rule{1ex}{0ex}}}(p_+ - p_-)^3 + O(1)(p_+- p_-)^4.
\end{equation}
By \eqref{eq:convex p}, the assumption $v_{pp}(p_\pm, S_\pm)>0$ implies that the condition of genuine nonlinearity holds and the shock is a Lax shock. Consider for example the case of a shock wave corresponding to the eigenvalue $\lambda_+$, which is characterized by $m>0$; for shock waves of the other family (i.e., corresponding to $\lambda_-$, then $m<0$) the proof is analogous. Since the shock is a Lax shock, then it is compressive, i.e., $p_+ < p_-$. By Lemma \ref{lem:entropy variation}, the entropy increases across the shock front if and only if $S_+ - S_- < 0.$ By \eqref{eq:entropy variation}, this inequality is implied once more by the condition $v_{pp}(p_\pm, S_\pm)>0$.
\end{proof}
\begin{nb}\label{rem:entropy variation}
The expression \eqref{eq:entropy variation} was first obtained by H. Bethe in 1942, see \cite[(3.44)]{Menikoff-Plohr} and \cite[\S 86]{Landau-Lifshitz}. This formula is notable because it depends neither on the particular equation of state nor on the form of internal structure. In particular, then, it holds for the equation of state \eqref{eq:pressure-ion} with \eqref{eq:saha-B}. 
\end{nb}
\par
Theorem \ref{thm:genuine nonlinearity} implies the following result.
\begin{theorem}\label{thm:entropy variation nonlinearity}
If either $\alpha \leq 60 \left(\frac{T}{\Ti}\right)^{\! 3}$ or $\frac{\Ti}{T} \leq 54.5375,$
then the entropy increases across a shock front whose amplitude $|p_+ - p_-|$ is sufficiently small.
\end{theorem}
\section{Integral curves}\label{sec:rarefaction curves}

We write system \eqref{eq:system} as $U_t + F(U)_x = 0$ or, for smooth solutions,
$$
U_t + A(U)U_x = 0,
$$
where $A(U) = F'(U)$. Let $V$ be a set of new unknowns related to $U$ by $U = U(V).$ We have
$$
   V_t + B(V)V_x = 0,
$$
where $B(V) = (U')^{-1}A(U)U'$ and $U'=U'(V)$ denotes the Jacobian matrix of $U$. Clearly, if $R(U)$ is a characteristic vector field of $A(U),$ then $\left(U'(V)\right)^{-1}R\left(U(V)\right)$ is a characteristic field of $B(V)$.

Some properties of system \eqref{eq:system} were already established in Lemma \ref{l:eigen} as far as the coordinates $(p,u,T)$ were concerned. In the following, the thermodynamic variables $\alpha, T$ are very useful. Recalling \eqref{eq:p-alphaT}, we consider the transformation of variables
$$
\left[\begin{array}{c}
            p\\
            u \\
            T
      \end{array}\right]
= \left[\begin{array}{c}
             \dfrac{1 - \alpha^2}{\kappa\alpha^2}T^{\frac{5}{2}}e^{-\frac{\Ti}{T}}\\
            u \\
            T
      \end{array}\right].
$$
With a slight abuse of notation, we still denote by $R_\pm$ and $R_0$, as in Lemma \ref{l:eigen}, the eigenvectors of the Jacobian matrix of the flux in coordinates $(\alpha,u,T)$.
\begin{proposition}[Characteristic fields]\label{prop:new R_{pm}}
The characteristic vector fields with respect to $(\alpha, u, T)$ coordinates are expressed as
\begin{equation}
 R_{\pm}
= \left[\begin{array}{c}
     \pm\frac{\frac{\alpha(1 - \alpha^2)}{2}}{\frac{5}{2} + \frac{1}{2}\alpha(1 - \alpha)\left(\frac52 + \frac{T_{\rm i}}{T}\right)^2}\frac{T_{\rm i}}{T}
    \\
     \frac{p}{\lambda}
    \\
   \pm\frac{1 + \frac{1}{2}\alpha(1 - \alpha)\left(\frac52 + \frac{T_{\rm i}}{T}\right)}{\frac{5}{2} + \frac{1}{2}\alpha(1 - \alpha)\left(\frac52 + \frac{T_{\rm i}}{T}\right)^2}\, T
  \end{array}\right],
\quad
R_0 =  \left[\begin{array}{c}
             \frac{\alpha(1 - \alpha^2)}{2}(\frac52 + \frac{T_{\rm i}}{T})\\
            0\\
            T
      \end{array}\right].
\label{eq:new_R}
\end{equation}
\end{proposition}
\begin{proof}
We compute
\[
          \frac{{\rm D}(p, u, T)}{{\rm D}(\alpha, u, T)}
= \left[\begin{array}{ccc}
            -\frac{2T^{\frac{5}{2}}}{\kappa \alpha^3}e^{-\frac{T_{\rm i}}{T}} &0 &  \frac{1 - \alpha^2}{\kappa\alpha^3}T^{\frac{3}{2}}\left(\frac52 + \frac{T_{\rm i}}{T}\right)e^{-\frac{T_{\rm i}}{T}}\\
            0 & 1 & 0\\
            0 & 0 & 1
      \end{array}\right]\quad
\left( \frac{T^{\frac{5}{2}}e^{-\frac{\Ti}{T}}}{\kappa\alpha^2}
 = \frac{p}{\alpha(1 - \alpha^2)}\right)
\]
and
\[
\left[\begin{array}{ccc}
            -\frac{2T^{\frac{5}{2}}}{\kappa \alpha^3}e^{-\frac{T_{\rm i}}{T}} &0 &  \frac{1 - \alpha^2}{\kappa\alpha^2}T^{\frac{3}{2}}\left(\frac52 + \frac{T_{\rm i}}{T}\right)e^{-\frac{T_{\rm i}}{T}}\\
            0 & 1 & 0\\
            0 & 0 & 1
      \end{array}\right]^{-1} = \left[\begin{array}{ccc}
            -\frac{\kappa \alpha^3}{2T^{\frac{5}{2}}e^{-\frac{T_{\rm i}}{T}}} &0 &  \frac{\alpha(1 - \alpha^2)}{2T}\left(\frac52 + \frac{T_{\rm i}}{T}\right) \\
            0 & 1 & 0\\
            0 & 0 & 1
      \end{array}\right].
\]
The result follows by what we pointed out just before the proposition.
\end{proof}

Now, we can consider integral curves. We denote with
\begin{equation}\label{eq:integral_curves}
\left(\alpha_{\pm}(s), u_{\pm}(s), T_{\pm}(s)\right)
\end{equation}
an integral curve of the characteristic field $R_{\pm}$ defined in \eqref{eq:new_R}, i.e., $(\dot\alpha_{\pm}(s), \dot u_{\pm}(s), \dot T_{\pm}(s))= R_\pm(\alpha,u,T)$. Indeed, we first focus on the projection of the integral curves on the $(\alpha,T)$-plane. We have, with a slight abuse of notation,
\begin{align}
\frac{\dot T(s)}{\dot \alpha(s)} = \frac{dT}{d\alpha} & = \frac{1 + \frac{1}{2}\alpha(1 - \alpha)\left(\frac52 + \frac{T_{\rm i}}{T}\right)}{\frac{\alpha(1 - \alpha^2)}{2}\frac{T_{\rm i}}{T}}\, T,
\label{eq:dT/dalpha}
\end{align}
and 
\begin{align}
    R_{\pm}\nabla \lambda_{\pm} & =   \pm\frac{\alpha(1 - \alpha^2)\frac{T_{\rm i}}{T}}{2\left[\frac{5}{2} + \frac{1}{2}\alpha(1 - \alpha)\left(\frac52 + \frac{T_{\rm i}}{T}\right)^2\right]}\left[\frac{\partial \lambda_{\pm}}{\partial \alpha} + \frac{1 + \frac{1}{2}\alpha(1 - \alpha)\left(\frac52 + \frac{T_{\rm i}}{T}\right)}{\frac{\alpha(1 - \alpha^2)}{2}\frac{T_{\rm i}}{T}}\, T\frac{\partial \lambda_{\pm}}{\partial T}\right]
 \nonumber
  \\
  &  = \pm\frac{\alpha(1 - \alpha^2)\frac{T_{\rm i}}{T}}{2\left[\frac{5}{2} + \frac{1}{2}\alpha(1 - \alpha)\left(\frac52 + \frac{T_{\rm i}}{T}\right)^2\right]}\frac{d\lambda_{\pm}}{d\alpha}.
\label{eq:dlambda/dalpha}
\end{align}
We observe that the ordinary differential equation \eqref{eq:dT/dalpha} does not depend on $u$ and can be uniquely solved by a function $T=T(\alpha)$ once the initial data $T(\alpha_0)=T_0$ is fixed. The projection on the $(\alpha,T)$-plane of any such characteristic curve is the same once $T_0$ is fixed. By inserting $T(\alpha)$ in the second equation $\dot u_\pm = R_2(\alpha,u_\pm,T)$ one finds $u_\pm=u_\pm(\alpha)$.

Recalling \eqref{eq:eta-alphaT}, with a slight abuse of notation we still denote
\begin{equation}\label{eq:HH0}
    \eta(\alpha,T)  =  - 2\log \frac{1 - \alpha}{\alpha} +  \left(\frac{5}{2} + \frac{\Ti}{T}\right)(1 + \alpha)
\end{equation}
as a representative for the dimensionless entropy. Since $\eta$ is a Riemann invariant for the characteristic fields $R_\pm$ by Lemma \ref{l:eigen}, then it is constant along the integral curves of those fields  in $(\alpha, u, T)$-space \cite[Theorem 7.6.2]{Dafermos}; of course, this can be directly verified from \eqref{eq:HH0}. In other words, the projections on the $(\alpha,T)$-plane of those integral curves coincide with the isentropes $\eta(\alpha,T) =$ const. 

\begin{proposition}[Integral curves in the $(\alpha,T)$-plane]\label{eq:integral curve alpha T}
For any $\eta_0\in\mathbb{R}$ the equation
\begin{equation}\label{eq:blowup_entropy}
2\log\left(\frac{1 - \alpha_\infty}{\alpha_\infty}\right) - \frac{5}{2}(1 + \alpha_\infty) + \eta_0 = 0
\end{equation}
has a unique solution $\alpha_\infty\in(0,1)$. The two integral curves of the characteristic fields $R_{\pm}$, corresponding to $\eta_0$, have the same projection on the $(\alpha,T)$-plane, which is the graph of a function $\mathcal{T}^{\eta_0}= \mathcal{T}^{\eta_0}(\alpha)$ for $\alpha\in(0,1)\setminus\{\alpha_\infty\}$. The function $\mathcal{T}^{\eta_0}$ is strictly increasing in both intervals $(0,\alpha_\infty)$ and $(\alpha_\infty,1)$; moreover,
\[
\lim_{\alpha\to0+} \mathcal{T}^{\eta_0}(\alpha) = 0,\quad
\lim_{\alpha\to\alpha_\infty\mp} \mathcal{T}^{\eta_0}(\alpha) = \pm\infty,\quad
\lim_{\alpha\to1-} \mathcal{T}^{\eta_0}(\alpha) = 0
\]
 and for $T =  \mathcal{T}^{\eta_0}(\alpha)$ we have the asymptotic expression
\begin{equation}\label{eq:rarefaction asymptotic}
    \alpha \sim e^{-\frac{T_i}{2T} + \frac{1}{2}\left(\eta_0 - \frac{5}{2}\right)} \quad
\text{as} \quad \alpha \to 0+.
\quad 
\end{equation}
\end{proposition}
\begin{proof}
Notice that the function $\eta$ can assume any real value when $\alpha$ varies in $(0,1)$. By solving the equation $\eta(\alpha,T)=\eta_0$ for $\frac{T}{\Ti},$ with $\eta$ given by \eqref{eq:HH0}, we have
\begin{equation}\label{e:def_mathcal_T}
\mathcal{T}^{\eta_0}(\alpha) := \frac{(1 + \alpha)\Ti}{2\log\left(\frac{1 - \alpha}{\alpha}\right) - \frac{5}{2}(1 + \alpha) + \eta_0} = \frac{(1 + \alpha)\Ti}{d(\alpha)}.
\end{equation}
The denominator $d(\alpha)$ of $\mathcal{T}^{\eta_0}$ is a strictly decreasing function of $\alpha\in(0,1)$; moreover, $\lim_{\alpha\to0+}d(\alpha)=+\infty$ and $\lim_{\alpha\to1-}d(\alpha) =-\infty$. Then, there is a unique value $\alpha_\infty\in(0,1)$ satisfying \eqref{eq:blowup_entropy}. 
The properties of the function $\mathcal{T}^{\eta_0}$ easily follow. 
Since
\[
    \frac{T_{\rm i}}{\mathcal{T}^{\eta_0}(\alpha) } \sim 2 \log \frac{1}{\alpha} - \frac{5}{2} + \eta_0
\quad
\text{as} \quad \alpha \to 0+, 
\]
the above asymptotic form is obtained.
\end{proof}

\begin{figure}[htbp]	
  \centering
  \begin{tabular}{c}
  \includegraphics[width =.65\linewidth]{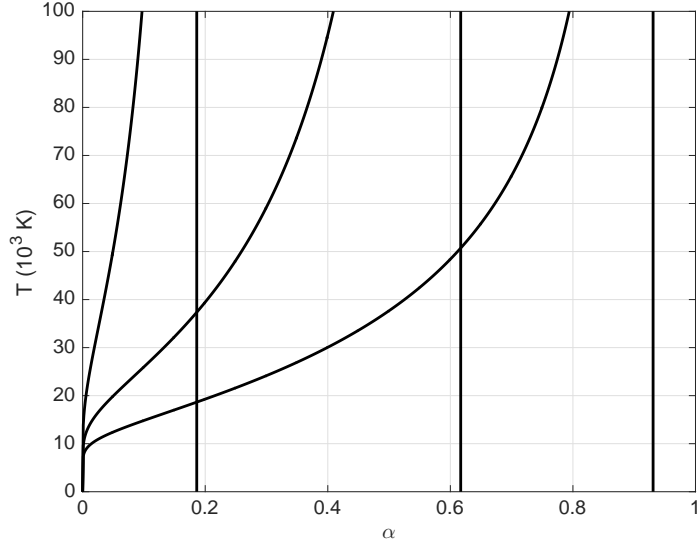}
  \end{tabular}
\caption{Integral curves of the characteristic fields $R_\pm$ in the $(\alpha,T)$ plane. We plotted the graphs of 
the functions $\mathcal{T}^{\eta_0}$ (right) for $\eta_0 = 0,5,10$; the vertical lines represent the asymptotes at $\alpha=\alpha_\infty$.}
  \label{fig:dg}
\end{figure}
\begin{nb}\label{r:Tnegative}
Notice that $\mathcal{T}^{\eta_0}(\alpha)>0$ for $\alpha \in (0,\alpha_\infty)$ and the value $\alpha_\infty$ only depends on $\eta_0$ and not on $\Ti$. Moreover, since the function $\mathcal{T}^{\eta_0}$ increases in the interval $(\alpha_\infty,1)$ and $\lim_{\alpha\to1-}\mathcal{T}^{\eta_0}(\alpha) = 0$, $\mathcal{T}^{\eta_0}(\alpha)<0$ for $\alpha \in (\alpha_\infty,1)$, which shows that the temperature is negative and  the states with $\alpha > \alpha_\infty$ are not physically admissible. Moreover, we remark that, by \eqref{eq:p-alphaT}, for $\alpha\to\alpha_\infty-$, on the integral curve we have $p(\alpha,T)\to\infty$  and, by \eqref{eq:pressure-ion} and \eqref{eq:p-alphaT}, we also have
$\rho = a(1+\alpha)\frac{T}{p} = a(1+\alpha)\kappa\frac{\alpha^2}{1-\alpha^2}T^{-\frac32}e^{\frac{\Ti}{T}}\to0$.

As a consequence, from now on we understand the function $\mathcal{T}^{\eta_0}$ to be defined only for $\alpha\in(0,\alpha_\infty)$; moreover, we drop the superscript $\eta_0$ when it is not strictly needed.
\end{nb}

Now, we study the convexity of the integral curves $\eta =$ const.

\begin{lemma}[Convexity of the integral curves]\label{lem:convexity eta}
The projection of any integral curve on the $(\alpha,T)$-plane is strictly convex for $\frac{\Ti}{T} \leq 4$ and strictly concave for small $\alpha.$
\end{lemma}
\begin{proof}
Instead of directly computing the derivatives of $\mathcal{T}^{\eta_0}$ from \eqref{e:def_mathcal_T} we exploit the implicit function theorem. We note that $\eta_\alpha = \frac{2}{\alpha(1 - \alpha)} + \left(\frac{5}{2} + \frac{\Ti}{T}\right)$, $\eta_T = - \frac{\Ti}{T^{2}}(1 + \alpha)$
and then
\begin{equation}\label{eq:eta''}
  \eta_{\alpha\alpha} = 2\frac{2\alpha -1}{\alpha^2(1 - \alpha)^2} , \qquad \eta_{\alpha T} = - \frac{\Ti}{T^2}, \qquad \eta_{TT} = \frac{2\Ti}{T^3}(1 + \alpha).
\end{equation}
Consequently, since $\mathcal{T}^{\eta_0}$ solves the implicit equation $\eta\left(\alpha, \mathcal{T}^{\eta_0}(\alpha_\infty)\right)=\eta_0$, we find that it satisfies
$$
  \frac{d^2\mathcal{T}^{\eta_0}(\alpha)}{d\alpha^2} =
-\left. \frac{\eta_{\alpha\alpha}\eta_T^2 - 2 \eta_{\alpha T}\eta_{\alpha}\eta_T + \eta_{TT}\eta_{\alpha}^2}{\eta_T^3}\right|{}_{\left(\alpha,\mathcal{T}^{\eta_0}(\alpha)\right)}.
$$
The denominator of the expression on the right-hand side is negative and, about the numerator, we find 
\begin{eqnarray*}
\lefteqn{\eta_{\alpha\alpha}\eta_T^2 - 2 \eta_{\alpha T}\eta_{\alpha}\eta_T + \eta_{TT}\eta_{\alpha}^2}
\\
&=& 2(1+\alpha)\frac{\Ti}{T^3}\left\{\frac{(2\alpha -1)(1 + \alpha)}{\alpha^2(1 - \alpha)^2}\frac{\Ti}{T} - \left[\frac{2}{\alpha(1 - \alpha)} + \frac{5}{2} + \frac{\Ti}{T}\right]\frac{\Ti}{T} + \left[\frac{2}{\alpha(1 - \alpha)} + \frac{5}{2} + \frac{\Ti}{T}\right]^2\right\}
\\
&=& \frac{2(1+\alpha)}{\alpha^2(1-\alpha)^2}\frac{\Ti}{T}\left\{-\left[1 - 3\alpha - \frac{5}{2}\alpha^2(1 - \alpha)^2\right]\frac{\Ti}{T} + 4\left[1 + \frac{5}{4}\alpha(1 - \alpha)\right]^{\! 2}\right\}.
\end{eqnarray*}
If $\frac{\Ti}{T} \leq 4$, then the numerator is positive and this proves the former part of the statement.
To prove the latter,  for $\alpha\to0$ we deduce 
$   
    \mathcal{T}^{\eta_0}(\alpha) \sim \frac{\Ti}{2\log \frac{1}{\alpha}}$ by \eqref{e:def_mathcal_T} 
and then the principal term of the numerator is $-  2\frac{\Ti^2}{T^4}\frac{1}{\alpha^2} < 0$. The lemma is completely proved.
%
\end{proof}

The next result deals with rarefaction waves. For brevity we call {\em 1-rarefaction curves} the rarefaction curves corresponding to the eigenvalue $\lambda_-$ while {\em 2-rarefaction curves} are those corresponding to the eigenvalue $\lambda_+$.

\begin{theorem}[Variation of the pressure along the integral curves]\label{thm:dp/dalpha>0}
We have
$$
  \frac{d}{d\alpha}p\left(\alpha,\mathcal{T}(\alpha)\right) > 0.
$$
Consider the projections on the $(\alpha,T)$-plane of the rarefaction curves corresponding to the eigenvalues $\lambda_\pm$ through the point $(\alpha_0,T_0)$; suppose that they are contained in the region where $R_{\pm}\nabla \lambda_{\pm}  > 0.$ Then the projection of the 1-rarefaction curve is the integral curve through the same point defined for $\alpha\in(0,\alpha_0]$ and that of the 2-rarefaction curve is the same integral curve defined for $\alpha \in [\alpha_0, \alpha_{\infty})$. 
\end{theorem}
\begin{proof}
By \eqref{eq:p-alphaT} we have
$
   \log p = \log (1 - \alpha^2) - 2 \log \alpha + \frac{5}{2}\log T - \frac{\Ti}{T} + \text{const.}
$
By differentiation together with \eqref{eq:dT/dalpha} we deduce
\begin{eqnarray*}
    \left.\frac{1}{p}\frac{dp}{d\alpha}\right|_{\left(\alpha, \mathcal{T}(\alpha)\right)}
&=& -\frac{2\alpha}{1 - \alpha^2} -\frac{2}{\alpha} + \left(\frac{5}{2\mathcal{T}(\alpha)} + \frac{\Ti}{\mathcal{T}(\alpha)^2}\right)\frac{d\mathcal{T}(\alpha)}{d\alpha}
\\
&=& \frac{2}{\alpha(1 - \alpha^2)}\left\{-1 + \left(1 + \frac{5}{2}\frac{\mathcal{T}(\alpha)}{T_{\rm i}}\right)\left[1 + \frac{1}{2}\alpha(1 - \alpha)\left(\frac52+\frac{T_{\rm i}}{\mathcal{T}(\alpha)}\right)\right]\right\}\ > \ 0.
\end{eqnarray*}
\par 
Since we are assuming $ R_{\pm}\nabla \lambda_{\pm}  > 0,$ then formula \eqref{eq:dlambda/dalpha} shows $\frac{d\lambda_{-}}{d\alpha} < 0,$ which yields that the 1-characteristic speed is strictly decreasing in $\alpha.$ Thus the 1-rarefaction curve corresponds to the part of the integral curve where $\alpha\in(0, \alpha_0].$ In the same way, since $\frac{d\lambda_{+}}{d\alpha} > 0,$ we conclude that the 2-rarefaction curve corresponds to the part $[\alpha_0, \alpha_{\infty}).$
\end{proof}
\begin{nb}\label{nb:backward 2-rarefacion curve}
The above theorem shows that the pressure is strictly increasing along the 2-rarefaction curve, which looks contradictory. However, here the state $(\alpha_0,T_0)$ is the {\em back} state of the 2-rarefaction wave; then the pressure is strictly {\em decreasing} along the 2-rarefaction curve from the front to the back state. 
\par
Theorem \ref{thm:genuine nonlinearity} shows that $ R_{\pm}\nabla \lambda_{\pm}  > 0$ in the region where $\alpha \leq 60 \left(T/\Ti\right)^{\! 3}.$ If the point  $(\alpha_0,T_0)$ is located in this region and $\alpha_0$ is sufficiently small, we conclude by \eqref{eq:rarefaction asymptotic} that the 1-rarefaction curve and the 2-rarefaction curve from the front to the back state stay in this region as $\alpha \to 0+.$
\end{nb}
We now continue the analysis of the integral curves for the characteristic fields $R_\pm$ by considering the behavior of the component $u$. By \eqref{eq:integral_curves} and \eqref{eq:new_R} we have
\begin{equation}
\frac{du}{d\alpha}  = \mp \frac{\frac52 +\frac12\alpha(1-\alpha)\left(\frac52 + \frac{T_{\rm i}}{T}\right)^2}{\frac{\alpha(1 - \alpha^2)}{2}\frac{T_{\rm i}}{T}}\,\frac{p}{\lambda}.
\label{eq:du/d alpha}
\end{equation}

\begin{proposition}[Integral curves in the $(\alpha,u)$-plane]\label{thm:integral curve alpha u}
The integral curves of the characteristic fields $R_{\pm}$ through a given point in the $(\alpha, u)$-plane are graphs of functions $u_\pm=u_\pm(\alpha)$ for $\alpha\in(0,\alpha_\infty)$. The function $u_+$ is decreasing while $u_-$ is increasing and they satisfy
\[
\lim_{\alpha\to0+}u_\pm(\alpha) = \hbox{ const.}, \quad \lim_{\alpha\to\alpha_\infty-}u_\pm(\alpha) = \mp\infty.
\]
Moreover,  we have
\[
u_\pm(\alpha) = O(1)\left(\frac{1}{\alpha_{\infty} - \alpha}\right)^{\!\frac{1}{2}}\quad \hbox{ for }\alpha\to\alpha_\infty-.
\]
\end{proposition}
\begin{proof}
We denote 
\[
\Lambda(\alpha,T) = \frac{\frac52 +\frac12\alpha(1-\alpha)\left(\frac52 + \frac{T_{\rm i}}{T}\right)^2}{\frac{\alpha(1 - \alpha^2)}{2}\frac{T_{\rm i}}{T}}\,\frac{p(\alpha,T)}{\lambda(\alpha,T)}.
\]
By \eqref{eq:du/d alpha}, for a fixed $\alpha_0\in(0,\alpha_\infty)$, the $u$-component of the integral curve of $R_\pm$ satisfying $u(\alpha_0)=0$ is
\begin{equation}\label{e:uLambda}
   u_\pm(\alpha) = \mp \int_{\alpha_0}^{\alpha} \Lambda\left(\beta, \mathcal{T}(\beta)\right)\, d\beta,
\end{equation}
where $\mathcal{T}(\alpha)$ is defined by \eqref{e:def_mathcal_T}, see Remark \ref{r:Tnegative}. We refer to Figure \ref{fig:uLambda} for graphs.

\begin{figure}[htbp]	
  \centering
  \begin{tabular}{c}
  \includegraphics[width =.65\linewidth]{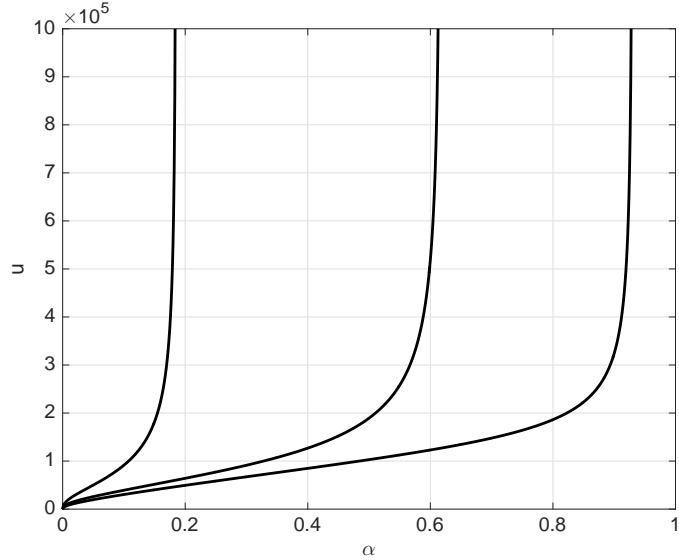}
  \end{tabular}
\caption{The integral curves \eqref{e:uLambda} of the characteristic field $R_-$ in the $(\alpha,u)$ plane corresponding to the parameters $\eta_0 = 0, 5, 10$, see \eqref{e:def_mathcal_T}. Here $\alpha_0\sim0$ and so $u(0) \sim0$. Units in mks.}
  \label{fig:uLambda}
\end{figure}

Since $\Lambda>0$, we deduce that $u_+$ is decreasing while $u_-$ is increasing in the interval $(0,\alpha_\infty)$. We notice that by \eqref{eq:c} we have
\begin{equation}\label{eq:p/c}
\frac{p}{\lambda} = a\sqrt{T(1 + \alpha)}\,\sqrt{\frac{\frac{3}{2} + \frac{1}{2}\alpha(1 - \alpha)\left((\frac52 + \frac{T_{\rm i}}{T})-1))^2 +\frac32\right)}{\frac{5}{2} + \frac{1}{2}\alpha(1 - \alpha)(\frac52 + \frac{T_{\rm i}}{T})^2}}.
\end{equation}
By \eqref{eq:p/c}, when $T\to\infty$ we have $\frac{p}{\lambda}=O(1)\sqrt{T}$ uniformly in $\alpha$ and then $\Lambda(\alpha,T) = O(1)T^{\frac{3}{2}}$, where the $O(1)$ term is not singular in $\alpha$. On the other hand we have, by \eqref{e:def_mathcal_T} and \eqref{eq:blowup_entropy}, that
\begin{align*}
d(\alpha) & = 2\log\left(\frac{1 - \alpha}{\alpha}\right) - \frac{5}{2}(1 + \alpha) - 2\log\left(\frac{1 - \alpha_\infty}{\alpha_\infty}\right) + \frac{5}{2}(1 + \alpha_\infty)
= O(1)(\alpha-\alpha_\infty) \quad \hbox{ for } \alpha\to\alpha_\infty-.
\end{align*}
Then $\Lambda\left(\alpha,\mathcal{T}(\alpha)\right) = O(1) \left(\frac{1}{\alpha_{\infty} - \alpha}\right)^{\!\frac{3}{2}}$ for $\alpha \to \alpha_{\infty}-$.
As a consequence,
$$
    u_\pm(\alpha) = \mp\int_{\alpha_0}^{\alpha} \Lambda\left(\beta,\mathcal{T}(\beta)\right) \,d\beta = O(1) \int_{\alpha_0}^{\alpha}\left(\frac{1}{\alpha_{\infty} - \beta}\right)^{\!\frac{3}{2}}\, d\beta = O(1) \left(\frac{1}{\alpha_{\infty} - \alpha}\right)^{\!\frac{1}{2}} \to\infty \quad \hbox{ for }\alpha \to \alpha_\infty-.
$$
When $T\to0$ we have again $\frac{p}{\lambda}=O(1)\sqrt{T}$ uniformly in $\alpha$ by \eqref{eq:p/c} and then $\Lambda(\alpha,T) = O(1)\frac{T^\frac32}{\alpha} + O(1)T^{-\frac{1}{2}}$ for $(\alpha,T)\to(0,0)$, where the $O(1)$ terms are not singular for $\alpha\to0+$. By \eqref{e:def_mathcal_T} we have
$$
    \mathcal{T}(\alpha) = O(1) \frac{1}{\log\frac{1}{\alpha}}\quad \hbox{ for }\alpha \to 0+,
$$
whence $\Lambda\left(\alpha,\mathcal{T}(\alpha)\right) = O(1) \sqrt{\log\frac{1}{\alpha}}$  for $\alpha \to 0+.$ Then
$$
    -u_\pm(\alpha) = \mp\int_{\alpha}^{\alpha_0} \Lambda\left(\beta,\mathcal{T}(\beta)\right) \,d\beta = O(1) \int_{\alpha}^{\alpha_0}\sqrt{\log\frac{1}{\beta}}\, d\beta = O(1) \int_{\frac{1}{\alpha_0}}^{\frac{1}{\alpha}}\frac{\sqrt{\log\beta}}{\beta^2}\, d\beta = O(1)\quad \hbox{ for }\alpha \to 0+.
$$
\end{proof}

\section{The High-Temperature-Limit model}\label{sec:HTLM}
In the case of very high temperatures one can use an approximate version of Saha's formula. More precisely, when $\frac{\Ti}{T}\sim 0$, we approximate Saha's law \eqref{eq:saha-B} with
\begin{equation}\label{eq:saha-B-HTLM}
  \frac{n_{r+1}n_{\rm e}}{n_r} = \frac{G_{r+1}g_{\rm e}}{G_r} \frac{(2\pi m_{\rm e} kT)^{\frac{3}{2}}}{h^3}.
\end{equation}
We call the corresponding model as the {\em High-Temperature-Limit (HTL) model}. By using \eqref{eq:saha-B-HTLM} instead of \eqref{eq:saha-B}, we see that formulas \eqref{eq:deg-ionization-rhoT} and \eqref{eq:deg-ionization} are replaced by 
\begin{equation}\label{eq:Saha HTLM}
\frac{\alpha^2}{1-\alpha} = \frac{\bar{\kappa}}{\rho}T^{\frac32},
\qquad    \alpha = \left(1 + \kappa pT^{-\frac{5}{2}}\right)^{-\frac{1}{2}},
\end{equation}
respectively. Moreover, we have
\begin{equation}\label{eq:p-HTL}
p = \frac{1 - \alpha^2}{\kappa \alpha^2}T^{\frac{5}{2}},\qquad v= \frac{R}{m}\frac1\kappa \frac{\alpha^2}{1-\alpha}T^{-\frac32}.
\end{equation}
We rewrite now the main results we obtained in the previous sections for the HTL model. 
%
%
%
%
The equation of state is again \eqref{eq:pressure-ion}. On the contrary, due to the approximation procedure, we discard the component due to the ionization energy in the internal energy; then, the internal energy, the enthalpy \eqref{eq:e} and the dimensionless entropy \eqref{eq:eta} become
\[
e = \frac32 \frac{R}{m}(1+\alpha)T,
\quad
H  =  \frac52 \frac{R}{m}(1+\alpha)T,
\quad
 \eta = -\log p + 2 \tanh^{-1} \alpha + \frac{5}{2}\alpha
  + \frac{5}{2} \log T + \text{\rm Const.}
\]
Notice that, by \eqref{eq:eta-alphaT}, the dimensionless entropy $\eta$ now explicitly depends {\em only on $\alpha$}:
\begin{equation}\label{eq:HTLM entropy}
 \eta =  - 2\log \frac{1 - \alpha}{\alpha} +   \frac{5}{2}\alpha  + \text{const}.
\end{equation}
Of course, $\eta$ depends on both $T$ and $p$ through $\alpha$. As a consequence, if $\alpha$ is constant, then clearly the entropy is constant. The following result is analogous to Lemma \ref{l:eigen}.
\begin{lemma}[Eigenvalues]\label{lem:HTLM characteristic alpha}
In the HTL model in Lagrangian coordinates, the characteristic speeds are
\begin{equation}\label{eq:HTLM characteristic speed}
  \lambda_{\pm}
 =  \pm\frac{p}{a} \sqrt{\frac{5}{3T(1 + \alpha)}}, \qquad \lambda_0=0.
\end{equation}
If we still denote by $R_\pm$ the eigenvectors corresponding to $\lambda_\pm$ as in \eqref{eq:eigenvectors}, we have $R_\pm\nabla \log \lambda_{\pm} = \pm\frac{4}{5p}$ and then both $\lambda_\pm$ are genuinely nonlinear. The eigenvalue $\lambda_0$ is linearly degenerate.
\end{lemma}
Notice that if $\alpha=0$ we formally find the eigenvalues of polytropic gas-dynamics with $\gamma=\frac53$. Also notice that the degree of ionization is constant along the $\pm$ characteristic directions, i.e.,  $R_{\pm}\nabla \alpha = 0$.
This means that both $\eta$ and $\alpha$ are Riemann invariants for the fields associated to $\lambda_\pm$; clearly, they are not independent by \eqref{eq:HTLM entropy}. A pair of Riemann invariants for $\lambda_0$ is still $\{u, p\}$.

Propositions \ref{prop:new R_{pm}}--\ref{thm:integral curve alpha u} and formula \eqref{e:def_mathcal_T} read now as follows.
\begin{proposition}[Integral curves]\label{prop:HTLM new R_{pm}}
In the HTL model, the characteristic vector fields $R_{\pm}$ with respect to $(\alpha, u, T)$ coordinates are expressed as
\begin{equation}\label{e:RpmR0}
R_{\pm} = \left[\begin{array}{c}
    0 \\
    \mp a\sqrt{\frac{3T(1 + \alpha)}{5}}\rule{0ex}{3.5ex}\\
   \frac{2T}{5}
  \end{array}\right],
\qquad
R_0 = \left[\begin{array}{c}
              \frac{5}{4}\alpha(1 - \alpha^2) \\
            0\\
            T
      \end{array}\right].
\end{equation}
The integral curves of the fields $R_\pm$ in the $(\alpha, T)$-plane are the straight lines
$\alpha = \text{\rm const.}$
More precisely, the integral curves through $(\alpha_0, u_0, T_0)$ are
$u - u_0 = \mp a\sqrt{15(1 + \alpha_0)}(T^{\frac{1}{2}} - T_0^{\frac{1}{2}})$, while the integral curves in the $(p, u, \alpha)$-space through $(p_0, u_0, \alpha_0)$ are
\begin{equation}\label{eq:HTLM integral p u}
	u - u_0 = \mp \sqrt{15(1 + \alpha_0)}a\left(\frac{\kappa \alpha_0^2}{1 - \alpha_0^2}\right)^{\!\frac{1}{5}}\left(p^{\frac{1}{5}} - p_0^{\frac{1}{5}}\right).
\end{equation}
\end{proposition}
\begin{proof}
The first two statements follow directly. About the formulas for integral curves, by \eqref{e:RpmR0} we have
$$
	\frac{du}{dT} = \mp \frac{a\sqrt{15(1 + \alpha_0)}}{2}T^{-\frac{1}{2}},
$$
whence the formula in the $(\alpha,u,T)$-space. At last, it follows from \eqref{eq:p-HTL} that along the integral curve of $R_{\pm}$ through $(p_0,u_0,\alpha_0)$ we have $p = \frac{1 - \alpha_0^2}{\kappa\alpha_0^2}T^{\frac{5}{2}}$;
hence, $T = \left(\frac{\kappa\alpha_0^2p}{1 - \alpha_0^2}\right)^{\!\frac{2}{5}}$
and \eqref{eq:HTLM integral p u} follows.
\end{proof}
\begin{nb}
The integral curves in the $(p, u, \alpha)$-space have a form similar to those of the polytropic gas with the adiabatic exponent $\gamma = \frac{5}{3}$, where the entropy replaces $\alpha$; see \cite{Liu5} or \cite[(2.7), (2.8)]{Asakura_basiclemmas}. Notice that in the HTL model the integral curves of the general model reduce to the asymptotes, see Figure \ref{fig:dg}. To emphasize the above correspondence we introduce the function 
$$
\mathcal{H}(\alpha) = \frac{2}{5}\log \alpha + \frac{1}{2}\log (1 + \alpha) -  \frac{1}{5}\log (1 -  \alpha^2).
$$
By using this pseudo-entropy, we have, see \cite{Smoller} for an analogous expression for $\frac53$-polytropic gasdynamics,
$$
   u - u_0 = \mp \sqrt{15}a\kappa^{\frac{1}{5}}e^{\mathcal{H}(\alpha_0)}\left(p^{\frac{1}{5}} - p_0^{\frac{1}{5}}\right).
$$
Compare
\[
  \eta(\alpha) = 2\log \alpha - 2\log (1 - \alpha) +   \frac{5}{2}\alpha\quad \hbox{ and }\quad
  \mathcal{H}(\alpha) = \frac{2}{5}\log \alpha + \frac{3}{10}\log (1 + \alpha) -  \frac{1}{5}\log (1 -  \alpha).
\]
\end{nb}

Now, we study the Hugoniot loci in the HTL model. The branch corresponding to contact discontinuities is given by $u=u_0$, $p=p_0$ and
\(
\alpha = \left(1 + \kappa p_0 T^{-\frac{5}{2}}\right)^{-\frac{1}{2}}.
\)
Then, we focus on shock waves. About the thermodynamic part of the Rankine-Hugoniot condition we have the following result, to be compared with Proposition \ref{prop:Hugoniot curve}; the proof is left to the reader.

\begin{lemma}[Variation of the temperature]\label{lem:HTLM T_+/T_-}
In the HTL model, the thermodynamic part of the Hugoniot locus \eqref{eq:Hugoniot-locus} is expressed as 
\begin{equation}\label{eq:T_+/T_-_HTLM}
    \frac{T(1 + \alpha)}{T_0(1 + \alpha_0)} = \frac{4 + \frac{p}{p_0}}{4 + \frac{p_0}{p}}
\end{equation}
or else as
\begin{equation}\label{eq:RH-HTLM}
T\left[4(1 + \alpha) + \left(\frac{1 - \alpha_0^2}{1 - \alpha}\right) \left(\frac{\alpha}{\alpha_0}\right)^2\left(\frac{T_0}{T}\right)^{\frac{5}{2}}\right]
=
T_0\left[4(1 + \alpha_0) + \left(\frac{1 - \alpha^2}{1 - \alpha_0}\right)\left(\frac{\alpha_0}{\alpha}\right)^2 \left(\frac{T}{T_0}\right)^{\frac{5}{2}}\right].
\end{equation}
\end{lemma}
The next result is analogous to that given in Proposition \ref{eq:Hugoniot curve alpha T}. An important difference that we emphasize in its proof is that the corresponding function $T=T(\alpha)$ is {\em not} differentiable at $\alpha_0$, where the tangent to its graph becomes vertical; see Figure \ref{fig:HugoniotT_HTML}.

\begin{proposition}[$T$ as a function of $\alpha$]\label{eq:HTLM Hugoniot curve alpha T}
In the $(\alpha,T)$-plane, the thermodynamic part of the Hugoniot locus of $(\alpha_0,T_0)$ is the graph of a strictly increasing function $T=T(\alpha)$, for $\alpha\in(0,1)$. Moreover, the function $T$ is differentiable for $\alpha\in(0,1)\setminus\{\alpha_0\},$ and tends to $0,\,\infty$ as $\alpha \to 0,\,1,$ respectively.
\end{proposition}
\begin{proof}
Arguing as to obtain \eqref{eq:dT/dalpha:p} we have formally
\begin{align}
\lefteqn{ \frac{dT}{d\alpha} \left[
 \left(\frac{1 - \alpha_0^2}{1 - \alpha }\right)\left(\frac{\alpha }{\alpha_0}\right)^{\!2}\left(\frac{T_0}{T}\right)^{\frac{5}{2}} +  \frac{5}{3}\left(\frac{1 - \alpha ^2}{1 - \alpha_0}\right)\left(\frac{\alpha_0}{\alpha }\right)^{\!2}\left(\frac{T}{T_0}\right)^{\frac{3}{2}} - \frac{8}{3}(1 + \alpha )\right]} \nonumber\\
&  =  \frac{4T_0\alpha_0^2}{3(1 - \alpha_0)}\left(\frac{1}{\alpha^3}\right)\left(\frac{T }{T_0}\right)^{\frac{5}{2}}
+  T\left[\frac{8}{3} + \frac{2(1 - \alpha_0^2)}{3\alpha_0^2}\left[\frac{\alpha(2 - \alpha)}{(1 - \alpha)^{2}}\right]\left(\frac{T_0}{T }\right)^{\frac{5}{2}}\right]. \label{eq:Hugoniot derivative HTLM}
\end{align}
By introducing $p$ and $p_0,$ we deduce
\begin{align}
\lefteqn{\frac{dT}{d\alpha} \left[ (1 + \alpha)\left(\frac{p_0}{p}\right)  + \frac{5}{3}(1 + \alpha_0)\left(\frac{p}{p_0}\right)\frac{T_0}{T} - \frac{8}{3}(1 + \alpha )\right]}
\nonumber 
\\
 &=  \frac{4}{3}\left(\frac{p}{p_{0}}\right)\frac{T_0(1 + \alpha_{0})}{\alpha(1 -  \alpha^{2})} + T\left[\frac{8}{3} + \frac{2}{3}\left(\frac{p_{0}}{p}\right)\frac{(1 + \alpha)(2 - \alpha)}{\alpha(1 - \alpha)} \right].\label{eq:dT/dalphaHTL}
\end{align}
Note that in both formulas the coefficient of $dT/d\alpha$ vanishes at $(\alpha_0,T_0)$, differently from \eqref{eq:dT/dalpha:p} where the exact Saha's law is used.
By \eqref{eq:T_+/T_-_HTLM} we deduce
\begin{equation}\label{eq:pp0HTL}
  \left(\frac{p}{p_0}\right)\frac{T_0}{T}(1 + \alpha_0)
= \left(\frac{p_0}{p}\right)(1 + \alpha) + 4(1 + \alpha) - 4(1 + \alpha_0)\frac{T_0}{T}.
\end{equation}
Then, formulas \eqref{eq:dT/dalphaHTL} and \eqref{eq:pp0HTL} imply
\begin{equation}
\frac{1}{T}\frac{dT}{d\alpha}
  = \frac{1 + \left(\frac{p_{0}}{p}\right)\frac{(1 + \alpha)(2 - \alpha)}{4\alpha(1 - \alpha)}  + \left(\frac{p}{p_{0}}\right)\frac{T_0(1 + \alpha_{0})}{2T\alpha(1 -  \alpha^{2})}}
{(1 + \alpha)\left(\frac{p_{0}}{p} - 1\right) + \frac{5}{2}\left[1 + \alpha - (1 + \alpha_0)\frac{T_0}{T}\right]}.
 \label{eq:HTLM Hugoniot derivative_3}
\end{equation}
We examine the denominator of \eqref{eq:HTLM Hugoniot derivative_3}, see \eqref{eq:F_T}, \eqref{e:Phi-first} and \eqref{eq:dT/dalpha:p}: to this aim we set
$$
	\Phi(\alpha, T) = (1 + \alpha)\left(\frac{p_{0}}{p} - 1\right) + \frac{5}{2}\left[1 + \alpha - (1 + \alpha_0)\frac{T_0}{T}\right].
$$
Let us recall that $\Phi(\alpha_0,T_0) = 0.$ 
We have
$$
  \frac{\partial \Phi}{\partial \alpha} (\alpha_0,T_0)
= \frac{2(1+\alpha_0)}{\alpha_0(1 - \alpha_0)} + \frac{5}{2}, \quad
\frac{\partial \Phi}{\partial T} (\alpha_0,T_0) = 0, \quad
\frac{\partial^2 \Phi}{\partial T^2} (\alpha_0,T_0) =  \frac{15}{4T_0^2}(1 + \alpha_0).
$$
Then, in a neighborhood of $(\alpha_0,T_0)$ the set $\left\{\Phi(\alpha,T) = 0\right\}$ is the graph of a function $\tilde\alpha=\tilde\alpha(T)$ and
\begin{equation}\label{e:alpha-alpha0}
  \tilde\alpha(T) - \alpha_0 \sim -\frac{15(1 + \alpha_0)}{8T_0^2\left[\frac{2(1+\alpha_0)}{\alpha_0(1 - \alpha_0)} + \frac{5}{2}\right]}(T - T_0)^2
\end{equation}
as $T\to T_0$. Let us denote the numerator  of \eqref{eq:HTLM Hugoniot derivative_3} by $\Psi(\alpha, T).$ Since $\Psi(\alpha_0,T_0) \neq 0,$  the thermodynamic part of the Hugoniot locus of $(\alpha_0,T_0)$ is represented in a small neighborhood of $(\alpha_0,T_0)$ by the graph of a function $\alpha =\alpha(T)$ and
$$
  T\frac{d\alpha}{dT} = \frac{\Phi(\alpha, T)}{\Psi(\alpha, T)}.
$$
Clearly,  $\frac{d\alpha}{dT}(T_0)  = 0$ and $\Phi_{T}(\alpha_0,T_0) = 0$ yields  $\frac{d^2\alpha}{dT^2}(T_0) = 0.$ Thus we conclude that  $\alpha -\alpha_0 = O(1)(T - T_0)^3$ and, by \eqref{e:alpha-alpha0},  the thermodynamic part of the Hugoniot locus at $(\alpha_0,T_0)$ stays in the region $\Phi\left(\alpha,T\right) > 0,\, \alpha\neq \alpha_0,$ in a neighborhood of $\alpha_0$; see Figure \ref{fig:HugoniotT_HTML} on the right.

\begin{figure}[htbp]	
  \centering
  \begin{tabular}{cc}
\includegraphics[width =.50\linewidth]{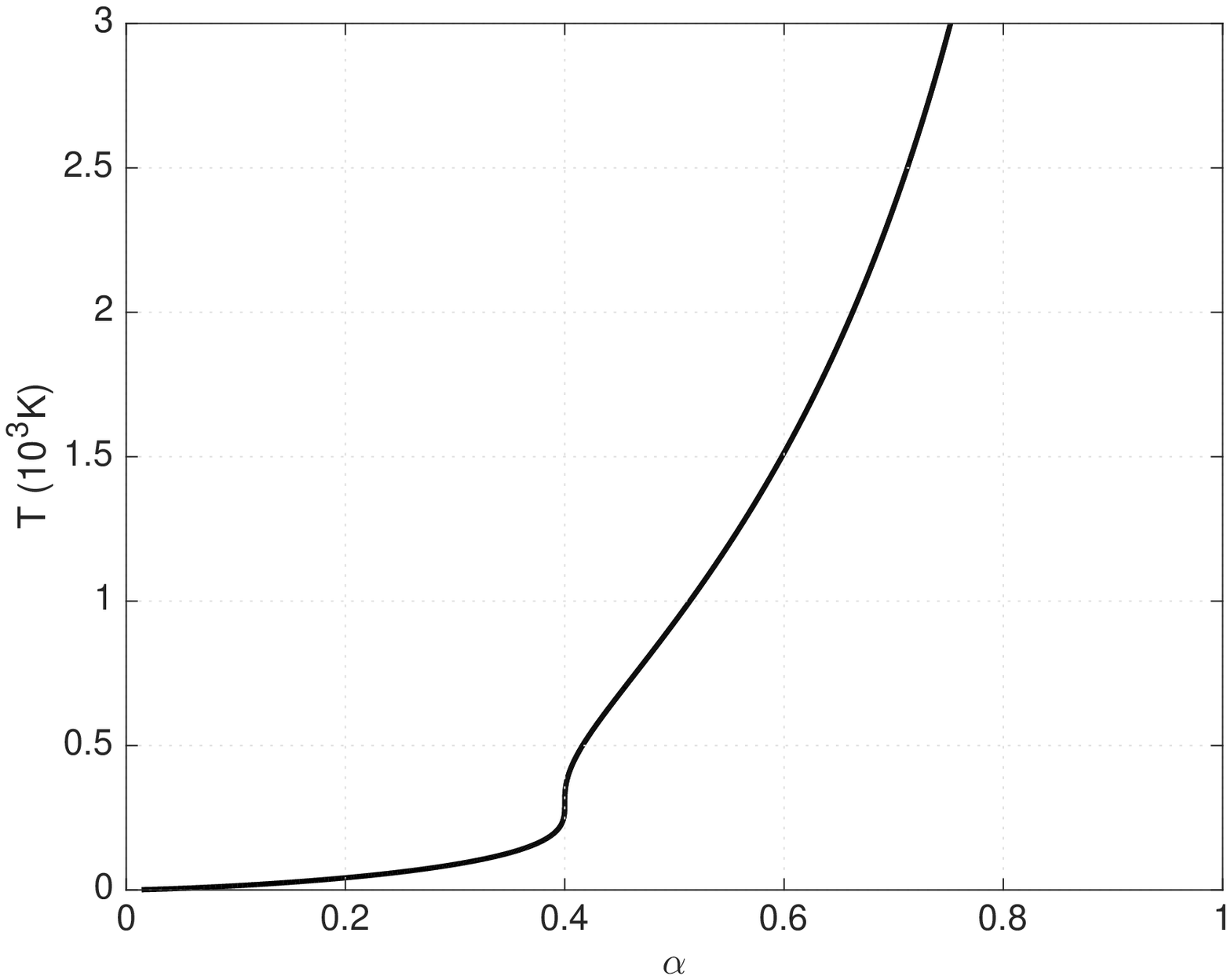}
&
\includegraphics[width =.50\linewidth]{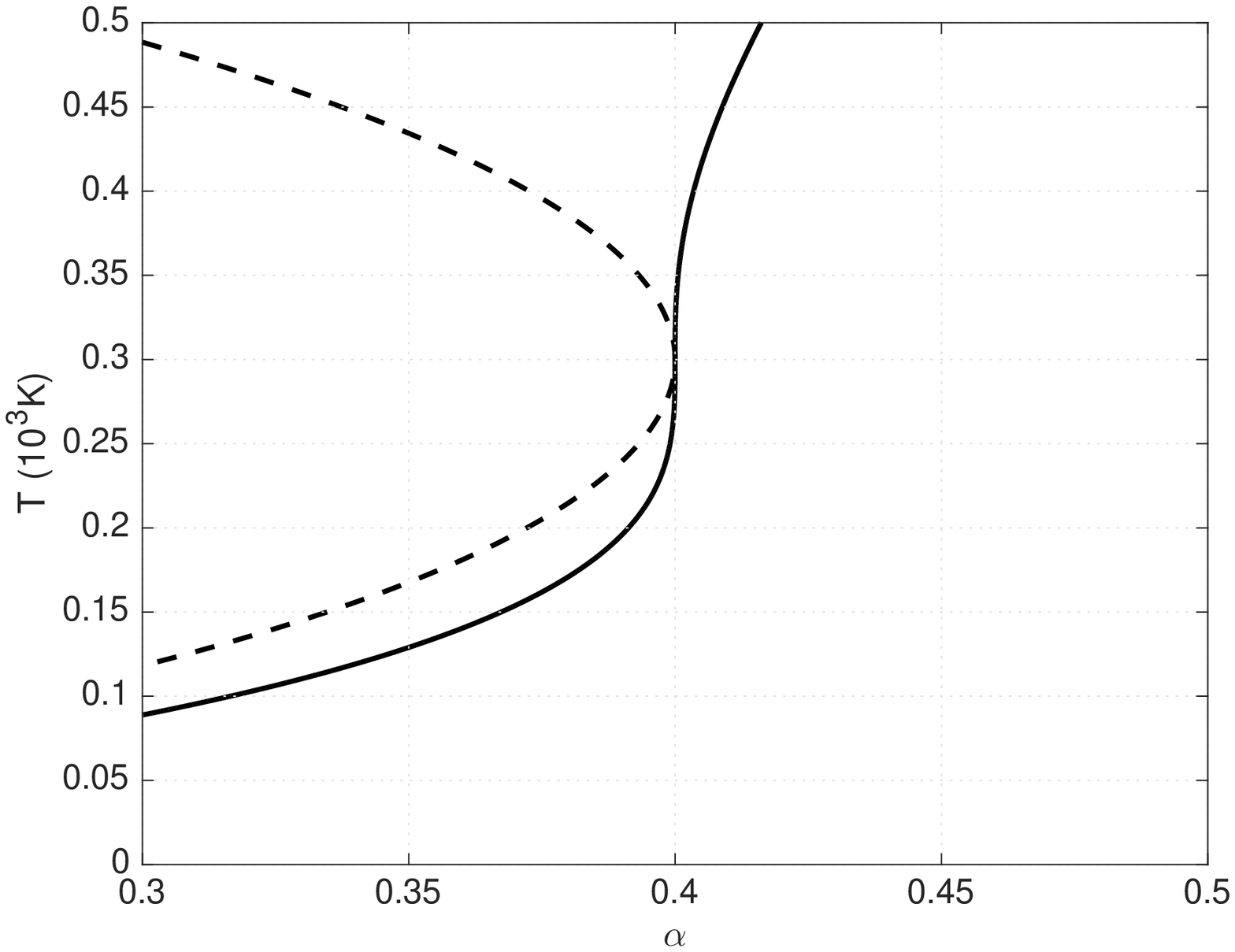}
  \end{tabular}
\caption{The thermodynamic part of the Hugoniot locus for the HTL model. On the left: the function $T=T(\alpha)$. On the right: a detail around $(\alpha_0,T_0)$ of the graph of the function $T=T(\alpha)$ and the zero set $\Phi(\alpha,T)=0$ (dashed line), see the proof of Proposition \ref{eq:HTLM Hugoniot curve alpha T}; the function $\Phi$ is positive on the right of the dashed line. Here $\alpha_0 = 0.4$, $T_0=300{\rm K}$ in order to emphasize the vertical tangent at $(\alpha_0,T_0)$.}
  \label{fig:HugoniotT_HTML}
\end{figure}

Let us denote $\Phi(\alpha) = \Phi\left(\alpha,T(\alpha)\right)$ with a slight abuse of notation. By \eqref{eq:HTLM Hugoniot derivative_3}, the ratio $\Pi := \frac{p_{0}}{p}$ satisfies the quadratic equation
$$
	(1 + \alpha)\Pi^{2} + 4\left[(1 + \alpha) - (1 + \alpha_{0}) \frac{T_{0}}{T}\right]\Pi - (1 + \alpha_{0}) \frac{T_{0}}{T}= 0.
$$
By setting $\Delta = (1 + \alpha) - (1 + \alpha_{0}) \frac{T_{0}}{T}$, we see that $\Pi$ is expressed as
$
	\Pi = \frac{\sqrt{4\Delta^{2} + (1 + \alpha)(1 + \alpha_{0})\frac{T_{0}}{T}} - 2\Delta}{1 + \alpha}
$;
hence,
\begin{eqnarray*}
\Phi(\alpha)
	&=& \sqrt{4\Delta^{2} + (1 + \alpha)(1 + \alpha_{0})\frac{T_{0}}{T}}  - (1 + \alpha) + \frac{1}{2}\Delta.
\end{eqnarray*}
Notice that $\Phi(\alpha_0)=0$. Suppose that $\frac{1}{2}\Delta > 1 + \alpha.$ Then $\Phi(\alpha) > 0$ and there is nothing to be proved. Assume hence $0 < \frac{1}{2}\Delta \leq 1 + \alpha.$  Then $\Phi(\alpha) = 0$ implies a quadratic equation for $\Delta$
\begin{equation} \label{eq:HTLM Delta}
 \textstyle  \frac{15}{4}\Delta^{2}    +  (1 + \alpha) \Delta
 + (1 + \alpha)(1 + \alpha_{0})\frac{T_{0}}{T} - (1 + \alpha)^{2}  = 0.
\end{equation}
Note that
$
	  (1 + \alpha)\Delta = (1 + \alpha)^{2} - (1 + \alpha)(1 + \alpha_{0})\frac{T_{0}}{T} .
$
Then we find that the equation \eqref{eq:HTLM Delta} turns out to be $\frac{15}{4}\Delta^{2} = 0,$ which is a contradiction.  Thus we have proved that $\Phi(\alpha) > 0$ for $\alpha \neq \alpha_0.$ 
\par 
Finally, by the same argument as proving Proposition \ref{eq:Hugoniot curve alpha T}, we can show that $T$ tends to $0,\,\infty$ as $\alpha \to 0,\,1,$ respectively. 
\end{proof}
\par
Next, we study the variation of both $p$ and $v$ along the Hugoniot locus, see Proposition \ref{p:dp/dalpha > 0}.

\begin{lemma}[$p$ and $v$ as functions of $\alpha$]\label{lem:dp/dalpha > 0HTL} Let $T=T(\alpha)$ be as in Proposition \ref{eq:HTLM Hugoniot curve alpha T}. Then
$$
    \frac{d}{d \alpha} p\left(\alpha, T(\alpha)\right)> 0\quad \text{ and } \quad \frac{d}{d \alpha} v\left(\alpha, T(\alpha)\right)> 0 \quad \text{for} \quad \alpha \geq \alpha_0.
$$
\end{lemma}
\begin{proof}
For simplicity we only deal with pressure. By letting $\frac{\Ti}{T} \sim 0$ in \eqref{eq:dpdalpha}, we find that $D(\alpha)$ is positive and $N(\alpha)$ is
\[
\frac{5}{4}\alpha(1 - \alpha^2)    +\frac{5(1 + \alpha)^2(2 - \alpha)}{16} \frac{p_{0}}{p} - \frac{3}{8}(1 + \alpha)\frac{p_0}{p} + (1 + \alpha )
	\ >\   - \frac{3}{8}(1 + \alpha)\frac{p_0}{p} + (1 + \alpha )\  >\ 0,
\]
for $p > p_0.$ Clearly the above expression is positive for all $p$ which are  close to $p_0.$ We conclude the proof by the same argument proving Proposition \ref{p:dp/dalpha > 0}.
\end{proof}

Now, we study the asymptotic behavior of the Hugoniot locus for extreme values. The proof is analogous to that of Proposition \ref{prop:asymptotic} and then is omitted. 

\begin{proposition}[Asymptotics]\label{prop:HTLM asymptotic}
On the Hugoniot locus \eqref{prop:asymptotic}, if $ T\to 0$ then $\alpha \to 0$ and if $ T\to \infty,$ then $\alpha \to 1$. More precisely, we have
$$
	\alpha \sim \frac{2\alpha_0}{\sqrt{1 - \alpha_0}} \left(\frac{T}{T_0}\right)^{\frac{3}{4}}\ \hbox{ for }T\to0\quad \hbox{ and }\quad 
	1 - \alpha \sim 4\frac{1 - \alpha_0}{\alpha_0^2} \left(\frac{T}{T_0}\right)^{-\frac{3}{2}} \ \hbox{ for }T\to\infty. 
$$
\end{proposition}
Now, we examine the kinetic part of the Hugoniot locus.

\begin{proposition}[The kinetic part of the Hugoniot locus]\label{lem:kinitic HTLM}
In the HTL model, the kinetic part of the Hugoniot locus of $(\alpha_0, u_0, T_0)$ is
\begin{equation}\label{e:kinetic-HTL}
        (u - u_0)^2 = \frac{3a^2 T_0(1 + \alpha_0)(p - p_0)^2}{p_0\left(4p + p_0\right)}.
\end{equation}
\end{proposition}
\begin{proof}
The kinetic part $\eqref{eq:Hugoniot-locus}_1$ of the Rankine-Hugoniot condition is $(u - u_0)^2 = -(p - p_0)(v - v_0)$. By \eqref{eq:spec-vol-ion} and Lemma \ref{lem:HTLM T_+/T_-} we have
$$
    \frac{v}{v_0} = \left(\frac{p_0}{p}\right)\frac{T(1 + \alpha)}{T_0(1 + \alpha_0)} = \left(\frac{p_0}{p}\right) \frac{4 + \frac{p}{p_0} }
{4 + \frac{p_0}{p}}
$$
and then \eqref{e:kinetic-HTL} follows by 
\begin{equation}\label{eq:kineticHTL}
   (u - u_0)^2 = 
   \frac{2v_0(p - p_0)^2}{\frac{8}{3}p + \frac{2}{3}p_0}.
\end{equation}
\end{proof}

Notice that \eqref{eq:kineticHTL} is the same formula for $\frac{5}{3}$-polytropic gasdynamics.

\section{Conclusions and discussions}\label{sec:conclusions}
In this paper we studied a model, consisting of three equations, for the macroscopic motion of an ionized gas in one space dimension. The model is closed by a state law and by a further thermodynamical relation called Saha's law. The degree of ionization $\alpha$ and the temperature $T$ are two proper independent thermodynamical variables. 
\par
We showed that these equations constitute a {\em strictly hyperbolic} system of conservation laws; this implies that the initial-value problem is well-posed locally in time for sufficiently smooth initial data. The geometric properties of the system are rather complicated and, to the best of our knowledge, have never been pointed out in the literature. In particular, we found the {\em loss of genuine nonlinearity} in a bounded region for both forward and backward characteristic fields.
\par
About the shock structure of the system, we found that the thermodynamical part of the Hugoniot locus of a fixed state $(\alpha_0,T_0)$ constitutes a simple graph in the $(\alpha,T)$-plane: more precisely, the temperature is an increasing function of the degree of ionization. Along this locus, also the pressure increases with the degree of ionization, which shows that the compressive branch of the locus is the part $\alpha \geq \alpha_0$. We notice that, even if the degree of ionization is small, the compressive branch may cross the region where the forward and backward characteristic fields are not genuinely nonlinear, which also causes the decrease of the physical entropy.
\par
Nevertheless, we established precise conditions which limit the \lq\lq classical\rq\rq\ region where the forward and backward characteristic fields are genuinely nonlinear and where the physical entropy is strictly concave. In such a region, the classical theory of shock waves is valid and the compression branch is admissible. In a forthcoming paper, we will discuss the electromagnetic shock tube filled with monatomic hydrogen gas and we expect that the actual experiments only involve such a classical region. 
\par
We also proposed the {\it High-Temperature-Limit model\/}, an approximation of the exact model in the case that the energy (or the temperature) is extremely high when compared with the dissociation energy. In this model, the physical entropy is just a function of the degree of ionization and does not depend on the temperature. Hence, the degree of ionization is a Riemann invariant. The integral curves (rarefaction curves) are similar to those of the monatomic ($\frac{5}{3}$-polytropic) gas. The shape of the thermodynamic part of the Hugoniot locus is quite different from that of the exact system considered in the previous sections.

\section*{Acknowledgments}
The second author is member of the {\em Gruppo Nazionale per l'Analisi Matematica, la Probabilit\'a e le loro Applicazioni} (GNAMPA) of the {\em Istituto Nazionale di Alta Matematica} (INdAM). He also was supported by the project {\em Balance Laws in the Modeling of Physical, Biological and Industrial Processes} of GNAMPA.

{\small
\bibliographystyle{abbrv}
\bibliography{refe-ion}
}
\end{document}